\documentclass[10pt, letterpaper, notitlepage, oneside, fleqn]{article}

\usepackage{palatino,url}
\usepackage{graphicx}
\usepackage{amsmath}
\usepackage{amsfonts}
\usepackage{mathrsfs}
\usepackage{amsthm}
\usepackage{amssymb}
\usepackage[all]{xy}
\usepackage{verbatim}
\usepackage{fancyhdr}
\usepackage{adjustbox,lipsum}
\usepackage{color}
\usepackage{longtable}
\usepackage{bbm}
\usepackage{tikz}

\usepackage{ytableau}
\ytableausetup{mathmode, boxsize=2em}

\fancyheadoffset{0cm}

\usepackage[vcentermath]{youngtab}
\newcommand{\Y}[1]{{\tiny\yng(#1)}}

\newcommand{\s}{\sigma}

\newcommand{\y}{\lambda}

\newcommand{\Ind}{\text{Ind}}

\newcommand{\Res}{\text{Res}}
\newcommand{\Hom}{\text{Hom}}

\newcommand{\End}{\text{End}}

\renewcommand{\dim}{\text{dim}}

\newcommand{\FIW}{\text{FI}_{\mathcal{W}}}
\newcommand{\FIBC}{\text{FI}_{BC}}

\newcommand{\FI}{\text{FI}}

\newcommand{\GL}{\text{GL}}
\newcommand{\SL}{\text{SL}}
\newcommand{\Sp}{\text{Sp}}
\newcommand{\SO}{\text{SO}}

\newcommand{\bn}{{\bf n}}
\newcommand{\bm}{ {\bf m}}

\newcommand{\W}{\mathcal{W}}

\newcommand{\oo}{\overline}

\newcommand{\C}{\mathbb{C}}

\newcommand{\Q}{\mathbb{Q}}

\newcommand{\Z}{\mathbb{Z}}
\newcommand{\F}{\mathbb{F}}

\renewcommand{\k}{\mathbbm{k}}

\newcommand{\G}{{\bf G}}
\newcommand{\B}{{\bf B}}
\newcommand{\Fr}{\text{Fr}}

\newtheorem{thm}{Theorem}[section]
\newtheorem{prop}[thm]{Proposition}
\newtheorem{lem}[thm]{Lemma}
\newtheorem{cor}[thm]{Corollary}

\theoremstyle{definition}
\newtheorem{defn}[thm]{Definition}
\newtheorem{rem}[thm]{Remark}
\newtheorem{example}[thm]{Example}

\newtheorem{notn}[thm]{Notation}

\theoremstyle{plain}

\date{\today}
\title{ Convergence criteria for $\FIW$--algebras and polynomial statistics on maximal tori in type B/C}

\author{Rita Jim\'enez Rolland\footnote{The first author is grateful for support from  the Sof\'ia Koval\'evskaia Foundation,  the Sociedad Matem\'atica Mexicana and PAPIIT-UNAM grant IA100816. $^\dagger$The second author is grateful for the support of an AMS-Simons Travel Grant.}  \;and Jennifer C. H. Wilson$^\dagger$}

\begin{document}


\setcounter{tocdepth}{2}

\maketitle

\begin{abstract} 

A result of Lehrer describes a beautiful relationship between topological and combinatorial data on certain families of varieties with actions of finite reflection groups. His formula relates the cohomology of complex varieties to point counts on associated varieties over finite fields.  Church, Ellenberg, and Farb use their \emph{representation stability} results on the cohomology of flag manifolds, together with classical results on the cohomology rings, to prove asymptotic stability for ``polynomial" statistics on associated varieties over finite fields. In this paper  we investigate the underlying algebraic structure of these families' cohomology rings that makes the formulas convergent. We prove that asymptotic stability holds in general for subquotients of  $\FIW$--algebras finitely generated in degree at most one, a result that is in a sense sharp.  As a consequence, we obtain convergence results for polynomial statistics on the set of maximal  tori in $\Sp_{2n}(\overline{\F_q})$ and $\SO_{2n+1}(\overline{\F_q})$ that are invariant under the Frobenius morphism. Our results also give a  new proof  of the stability theorem for invariant maximal tori in $\GL_n(\overline{\F_q})$ due to Church--Ellenberg--Farb.
\end{abstract}

\tableofcontents

\section{Introduction}

In this paper we study the complete complex flag varieties associated to the linear groups $\SO_{2n+1}(\C)$  and $\Sp_{2n}(\C)$ of type B and C, respectively. These spaces and their cohomology algebras are described in Section \ref{SectionMaxToriBackground}; their cohomology admits actions of the signed permutation group $B_n$.  A remarkable formula due to Lehrer  \cite{LEHRER_TORI} relates the representation theory of these cohomology groups with point-counts on varieties over $\F_q$ that parametrize maximal  tori in $\SO_{2n+1}(\overline{\F_q})$  and $\Sp_{2n}(\overline{\F_q})$. One result of this paper is to exhibit a form of convergence for these formulas as the parameter $n$ tends to infinity. 

 Church, Ellenberg, and Farb prove that the cohomology groups of the flag varieties  associated to $\GL_n(\C)$ are \emph{representation stable} with respect to the $S_n$--action \cite[Theorem 1.11]{CEF}. Using this result and a description of the cohomology due to Chevalley, they establish asymptotic stability for ``polynomial'' statistics on the set of Frobenius-stable maximal tori in $\GL_n(\overline{\F_q})$  \cite[Theorem 5.6]{CEFPointCounting}.  In Theorem \ref{ASYM_TORI} we prove the corresponding result in type B and C, using a representation stability result of the second author \cite{FIW1}.  
 
 To establish our results, we give general combinatorial criteria on the cohomology algebras of the complex varieties that ensure convergence. This is one of few results in the $\FI$--module literature that makes full use of the multiplicative structure on the $\FIW$--algebras. Church--Ellenberg--Farb prove convergence for statistics on maximal tori in type A using sophisticated results specific to these cohomology algebras. This paper shows that convergence in fact follows from much simpler features of the algebras.

 \begin{table}[h!]\label{SPECIFIC}	
\begin{footnotesize}
\caption{Some  statistics for   Frobenius-stable maximal tori of $\Sp_{2n}( \overline{\F_q})$ and $\SO_{2n+1}(\overline{\F_q})$}
\begin{longtable}{c  c  c  c   } 
	
	{\bf $\Fr_q$-stable maximal tori statistic} & {\bf Hyperoctahedral} & {\bf Formula in } &  {\bf Limit} \\
{\bf for $\mathbf{\Sp_{2n}( \overline{\F_q})}$ and $\mathbf{\SO_{2n+1}(\overline{\F_q})}$ } & {\bf character} & {\bf terms of $\mathbf{n}$}& {\bf  as  $\mathbf{n\to\infty}$} \\  \hline \\
		
Total number of    
  & $1$  & $q^{2n^2}$ (Steinberg)\\ 
		$\Fr_q$-stable maximal tori &  &  &  \\
		&&&\\
			Expected number of   & $X_1+Y_1$  & $ \displaystyle 1+\frac{1}{q^2}+\frac{1}{q^4}+\cdots \frac{1}{q^{2n-2}}$ & $  \displaystyle \frac{q^2}{(q^2-1)}  $\\ 
					 $1$-dimensional $\Fr_q$-stable subtori & & &  \\
	&&&\\
			
			Expected number of  split  & $X_1$  & $\displaystyle  \frac{1}{2}\left(1+\frac{1}{q}+\frac{1}{q^2}+\cdots \frac{1}{q^{2n-1}}\right) $  & $  \displaystyle   \frac{q}{2(q-1)}  $\\ 
					 $1$-dimensional $\Fr_q$-stable subtori & & &  \\
		&&&\\
		
		Expected value of reducible  &  $ \displaystyle  \binom{X_1+Y_1}{2}- (X_2+Y_2)$  &  $\displaystyle \frac{\left( q^4 - \frac{1}{q^{2n}} \right) \left( 1 - \frac{1}{q^{2(n-1)}} \right) }{\left(q^2-1 \right)\left(q^4 -1\right)} $&  $\displaystyle \frac{q^4}{(q^2-1)(q^4-1)}$ \\				 minus	irreducible $\Fr_q$-stable & &&\\ 		$2$-dimensional subtori & &&\\

&&&\\		
		Expected value of split minus & $ \displaystyle  X_2-Y_2$  &  $\displaystyle \frac{ q^2 \left(1-\frac1{q^{2n}}\right)\left(1-\frac{1}{q^{2(n-1)}}\right)}{2\left(q^4-1\right)}$&  $\displaystyle \frac{ q^2}{2\left(q^4-1\right)}$ \\ 
				 non-split $\Fr_q$-stable $2$-dimensional  &   & &   \\
		irreducible subtori &  &  & \\
\hline
	\end{longtable}
	\end{footnotesize}
	\end{table}

 Examples of limiting statistics computed in Section \ref{SectionConcreteStatistics} are presented in Table \ref{SPECIFIC}; the notation in this table is defined below. The first of these result  is due to Steinberg  \cite[Corollary 14.16]{STEIN}, and the other results appear to be new.

\subsubsection*{A general convergence result: algebras generated in low degree}

A primary goal of this paper is to describe the underlying algebraic structure on the cohomology algebras that drives these stability results. Throughout, we will use the notation $\W_n$ to denote either the family of symmetric groups $S_n$  (the Weyl groups in type A) or the family of hyperoctahedral groups $B_n$ (the Weyl groups in type B/C). We prove in Section  \ref{PROOFMAIN} that convergence of the statistics on maximal tori in type A and B/C holds because the graded pieces of the associated coinvariant algebras are, as a sequence of graded $\W_n$--representations, subquotients of what we call \emph{an $\FIW$--algebra finitely generated in degree at most 1}.  Terminology and basic theory of $\FIW$--modules is summarized in Section \ref{SubsectionBackgroundFIW}. Speaking informally, these $\FIW$--algebras are sequences of graded algebras with $\W_n$--actions that are generated by representations of $\W_0$ and $\W_1$. Concretely, they include the sequence of algebras $\Gamma_{\W_n}^*$  which we now define. 

Let $\k$ be a subfield of $\C$. Let $M({\bf 1})_n$ denote the $\k$--vector space with basis $x_1, x_2, \ldots, x_n$. The symmetric group $S_n$ acts on the variables $x_i$ by permuting the subscripts. Let $M({\bf 0})_n$ denote the $\k$--vector space generated by the single element $y$ with trivial $S_n$--action. Then for any integers $b, c \geq 0$ let $\Gamma^*_{S_n}$ be the graded algebra generated by $$ M({\bf 0})_n^{\oplus b} \oplus M({\bf 1})_n^{\oplus c} $$ with each copy of  $M({\bf 0})_n$ and $M({\bf 1})_n$  assigned a positive grading. The algebra $\Gamma^*_{S_n}$ may be taken to be commutative, anticommutative, or graded-commutative. It inherits a diagonal action of $S_n$ on its monomials. 

Analogously, let $M_{BC}({\bf 1})_n$ be the $\k$--vector space with basis $x_1, x_{\oo1}, x_2, x_{\oo2}, \ldots, x_n, x_{\oo{n}} $. The hyperoctahedral group $B_n$ acts on the variables $x_i$ by permuting and negating the subscripts. Again let $M_{BC}({\bf 0})_n$ be the $\k$--vector space generated by a variable $y$ with trivial $B_n$ action. Let $\Gamma^*_{B_n}$ be the symmetric, exterior, or graded-commutative algebra generated by $$ M_{BC}({\bf 0})_n^{\oplus b} \oplus M_{BC}({\bf 1})_n^{\oplus c} $$ with each family of variables graded in positive degree. 
 
Theorem \ref{LIMIT} shows the algebras $\Gamma^*_{\W_n}$ and their subquotients satisfy a convergence property with respect to certain family of `polynomial' class functions, first defined on the symmetric groups by MacDonald \cite[I.7 Example 14]{MacdonaldSymmetric} as follows. For any permutation $\sigma$ and positive integer $r$, let $X_r$ be the function that outputs  the number $X_r(\sigma)$  of $r$--cycles in the cycle type of $\sigma$. Polynomials in the functions $X_r$, called \emph{character polynomials}, play a major role in the theory of representation stability developed by Church--Ellenberg--Farb \cite{CEF}. In Section \ref{SubsectionCharacterPolys} we recall an analogous definition for the Weyl group $B_n$ in type B and C.  These \emph{hyperoctahedral character polynomials} are polynomials $$P \in \k[X_1, Y_1, X_2, Y_2, \ldots ]$$ in the  `signed cycle counting functions'  $X_r$ and $Y_r$ on $B_n$. See Section \ref{SubsectionRepTheoryBn} for the prerequisite background on the representation theory of the groups $B_n$ and a precise definition of these class functions. 

Theorem \ref{LIMIT} is a stability result for the asymptotics of the inner products $ \langle P_n, A^d_n \rangle_{\W_n}$ of a character polynomial $P$ with a subquotient $A^d_n$ of $\Gamma^d_{\W_n}$.  This is precisely the stability result needed to prove convergence of the point-counts that appear in Lehrer's formulas.  \\

\noindent {\bf Theorem \ref{LIMIT}} {\bf (Criteria for convergence).}  {\it Let $\W_n$ denote one of the families $S_n$ or $B_n$.  Let $\Gamma^*_{\W_n}$ be one of the commutative, exterior, or graded commutative algebras defined above, and let $\{A^d_n\}$ be any sequence of graded $\W_n$--representations such that $A^d_n$ is a subrepresentation of $\Gamma^d_{\W_n}$. 
Then for any $\W_n$ character polynomial $P$ and integer $q>1$ the following series converges absolutely
$$ \sum_{d=0}^{\infty}  \frac{ \lim_{n \to \infty} \langle P_n, A^d_n \rangle_{\W_n}}{q^d}. $$ 
\\}

\noindent The proof (given in Section \ref{PROOFMAIN}) uses an analysis of the combinatorics of coloured partitions. A closely related result in Type A is proved in \cite{FarbWolfson} using very different methods; see Remark \ref{RemarkFarbWolfson}. 

\subsubsection*{A failure of convergence}

In forthcoming work \cite{JIMWIL2} the authors show that, in contrast to Theorem \ref{LIMIT}, convergence may fail for $\FIW$--algebras generated in degree $2$ or greater. To give a concrete example, let $A^*_n$ be the polynomial algebra $\k[x_{1,2}, x_{1,3}, x_{2,3}, x_{1,4}, \ldots, x_{n-1, n} ]$ generated by commuting variables $x_{i,j} = x_{j,i}$,  $i \neq j$, in graded degree $1$, with an action of $\W_n$ by permuting the indices. Then the series $$\sum_{d=0}^{\infty}  \frac{ \lim_{n \to \infty} \langle P_n, A^d_n \rangle_{\W_n}}{q^d}$$ does not converge for any positive integer $q$ even for  the constant character polynomial $P=1$. This counterexample shows a fundamental difference in the asymptotic behaviour of $\FIW$--algebras generated by representations of  $\W_0$ or $\W_1$ and those generated by representations of $\W_n$ in degree $n \geq 2$. 

\subsubsection*{An application: polynomial statistics for coinvariant algebras and $\Fr_q$-stable maximal  tori }

 We apply  Theorem \ref{LIMIT} to the family $\{ R^d_n\}$, where $R^d_n$ denotes the   $d^{th}$-graded piece  of the complex coinvariant algebra $R^*_n$ in type A or in type  B/C. A description of these algebras appears in Section \ref{SubsectionCoinvariantAlg}. We obtain a convergent formula in Proposition \ref{CONVCOIN}  for coinvariant algebras of type B/C and give a new proof of the result in type A (see \cite[Theorem 5.6]{CEFPointCounting} and Remark \ref{COTYPEA}). Theorem \ref{LIMIT} shows that the convergence of this formula follows because the complex  coinvariant algebras are quotients  of the polynomial rings $\C[x_1, x_2, \ldots, x_n]$; convergence does not depend on any deeper structural features of these algebras. 

For a general connected reductive group $\G$ defined over $\F_q$ ,  let $\mathcal{T}^{\Fr_q}$  denote  the set of maximal tori of $\G(\overline{\F_q})$   that are stable under the action of $\Fr_q$, the Frobenius morphism.  $\Fr_q$-stable maximal tori of reductive groups $\G$ defined over $\F_q$ are fundamental to the study of the representation theory of the finite group $\G^{\Fr_q}$  \cite{DELUSZ}. More background is given in Section \ref{SectionMaxToriBackground}. Lehrer  \cite{LEHRER_TORI} obtained  remarkable formulas that relate functions that are defined on the set $\mathcal{T}^{\Fr_q}$ with the character theory of the Weyl group of $\G$.  We recall one of Lehrer's formulas in Theorem \ref{LEHFOR} for the cases when $\G$ is the symplectic group  $\Sp_{2n}$ or the special orthogonal group  $\SO_{2n+1}$; the formula connects functions on $\mathcal{T}^{\Fr_q}$ and  the cohomology of the generalized flag varieties associated to $\G(\C)$.

 In Section \ref{SubsectionToriPolys} we describe how the hyperoctahedral character polynomials of Section \ref{SubsectionCharacterPolys} define functions on $\mathcal{T}^{\Fr_q}$.  Specifically, for $T\in\mathcal{T}^{\Fr_q}$, $X_r(T) $ is the number of $r$-dimensional  $\Fr_q$-stable subtori of $T$ irreducible over $\F_q$ that split over $\F_{q^r}$, and $Y_r(T)$ is the number of $r$-dimensional $\Fr_q$-stable subtori of $T$ irreducible over $\F_q$ that do not split over $\F_{q^r}$. In Section \ref{MAXTORI},  we combine Lehrer's formula with Proposition \ref{CONVCOIN} to establish asymptotic results for these polynomial statistics  on the set of $\Fr_q$-stable maximal  tori of $\Sp_{2n}(\overline{\F_q})$ and of  $\SO_{2n+1}(\overline{\F_q})$.  We obtain Theorem \ref{ASYM_TORI} below, the type B and C analogues to  \cite[Theorem 5.6]{CEF} for $\Fr_q$-stable maximal tori for  $\GL_n(\overline{\F_q})$. \\

\noindent {\bf Theorem \ref{ASYM_TORI}.} {\bf (Stability of maximal tori statistics).} {\it
		Let $q$ be an integral power of a prime $p$. For $n\geq 1$, denote by $\mathcal{T}_n^{\Fr_q}$ the set of $\Fr_q$-stable maximal tori for either  $\Sp_{2n}(\overline{\F_p})$ or  $\SO_{2n+1}(\overline{\F_p})$. Let  $R^d_m$ denote the $d^{th}$-graded piece  of the complex coinvariant algebra $R^*_m$ in type B/C.
 If $P \in \C[X_1, Y_1, X_2, Y_2, \ldots ]$ is any hyperoctahedral character polynomial, then the normalized statistic  $$q^{-2n^2} \sum_{T \in \mathcal{T}_n^{\Fr_q}} P(T)$$ converges as $n\to\infty$. In fact,  $$ \lim_{n \to \infty} q^{-2n^2} \sum_{T \in \mathcal{T}_n^{\Fr_q}} P(T) \; = \; \sum_{d=0}^{\infty} \frac{\lim_{m \to \infty} \langle P_m, R^d_m \rangle_{B_m} }{q^d},$$
and the series in the right hand converges. \\ }

At the end of Section \ref{MAXTORI}, we use Theorem \ref{ASYM_TORI} and  Stembridge's formula for decomposing the $B_n$--representation $R^d_n$ (see Theorem \ref{ThmStembridge}) to  compute  the specific asymptotic counts in  Table \ref{SPECIFIC}.

\subsubsection*{Related work} Fulman \cite[Section 3]{FULMAN}  uses generating functions to recover the specific counts obtained by Church--Ellenberg--Farb \cite{CEFPointCounting}  for $\Fr_q$-stable maximal tori in $\GL_n(\overline{\F_q})$. Chen uses generating function techniques to obtain a linear recurrence on the stable twisted Betti numbers of the associated flag varieties\cite[Corollary 2 (II)]{WEIYAN}. These  methods  should also apply to study maximal tori of other classical groups, including the ones considered in this paper. They offer an alternate approach to performing the computations in Section \ref{SectionConcreteStatistics}, and may shed further light on the structure of the asymptotic formulas. 

Farb and Wolfson \cite{FarbWolfson}  extend the methods of Church--Ellenberg--Farb \cite{CEFPointCounting} to establish new algebro-geometric results. In \cite[Theorem B]{FarbWolfson} they prove  that configuration spaces of $n$ points in smooth varieties have \emph{convergent cohomology} \cite[Definition 3.1]{FarbWolfson}  and obtain arithmetic statistics for configuration spaces  over finite fields \cite[Theorem C]{FarbWolfson}. 

The results of this paper and the work of Church--Ellenberg--Farb \cite{CEFPointCounting} complement the existing literature on enumerative properties of matrix groups over finite fields; see for example  \cite{NP95, NP98, F99, W99, NP00,B,  FNP, LNP,NP,   NPP, NPS, BGPW, NPP14}. These authors seek, for example, to determine the proportion of elements in certain finite matrix groups that are \emph{cyclic}, \emph{separable}, \emph{semisimple}, or \emph{regular}. Their methods are different from those in this paper, and include the calculus of generating functions, complex analysis, classical Lie theory, finite group theory, analytic number theory, and statistics. This probabilistic approach to the study of finite groups has applications to algorithm design in group theory, random matrix theory, and combinatorics.

\subsubsection*{Acknowledgments} We would like to thank  Benson Farb for proposing this project. We thank Weiyan Chen, Tom Church, and Jason Fulman for useful conversations. We are grateful to  Ohad Feldheim for suggesting the proof of  Proposition \ref{TNCSubexponential}. We thank Jason Fulman for finding an error in a formula in Table \ref{SPECIFIC} in an early version of this paper. The first author would also like to thank the Centro de Ciencias Matem\'aticas, UNAM-Morelia, for all its support during the writing of this paper.

\section{Preliminaries}\label{PRE}

In this section we summarize some necessary background material and terminology. Section \ref{SubsectionRepTheoryBn} describes the representation theory of the hyperoctahedral groups in characteristic zero. In Section \ref{SubsectionBackgroundFIW} we review the foundations of $\FIW$--modules developed by Church--Ellenberg--Farb \cite{CEF} in type A and by the second author \cite{FIW1, FIW2} in type B/C. Section \ref{SubsectionBackgroundAsymptotics} defines the asymptotic notation that will be used in this paper. 

\begin{notn}{\bf (The inner product $\langle - \, ,  \, - \rangle_{G}$).} Throughout the paper, given a finite group $G$ we write $\langle - \, ,  \, - \rangle_{G}$ to denote the standard inner product on the space of class functions of $G$. By abuse of notation we may write a $G$--representation $V$ in one or both arguments to indicate the character of $V$.
\end{notn}

\subsection{Representation theory of the hyperoctahedral group $B_n$} \label{SubsectionRepTheoryBn}

Let $B_n$ denote the Weyl group in type B$_n$/C$_n$, the wreath product $$ B_n \cong \Z/2\Z\wr S_n=(\Z/2\Z)^n\rtimes S_n,$$ which we call the \emph{hyperoctahedral group} or \emph{signed permutation group}. By convention $B_0$ is the trivial group. We may view $B_n$ as the subgroup of permutations $S_{\Omega}\cong S_{2n}$ on the set $\Omega = \{ 1, \oo1, 2, \oo2, \ldots, n, \oo{n}\}$ defined by $$ B_n = \{ \sigma \in S_{\Omega} \; \; | \; \; \sigma(\oo{a}) = \oo{\sigma(a) }\quad \text{for all $1 \leq a \leq n$} \; \}.$$  Here $\oo{a}$  denotes negative $a$; in general the bar represents the operation of negation and satisfies $\oo{\;\oo{a}\;}=a$ for $a \in \Omega$.

\subsubsection{ Conjugacy classes of the hyperoctahedral groups}
 The conjugacy classes of $B_n$ are classified by \emph{signed cycle type}, defined as follows.  A \emph{positive $r$--cycle} is a signed permutation of the form $(s_1 s_2 \cdots  s_r)(\oo{s_1} \, \oo{s_2} \cdots \oo{s_r}) \in S_{\Omega}$, $s_i \in \Omega$. This element reverses the sign of an even number of letters. A \emph{negative $r$--cycle} is a signed permutation of the form $(s_1 s_2 \cdots s_r \oo{s_1} \cdots \oo{s_r})  \in S_{\Omega}$, which reverses an odd number of signs. The $r^{th}$ power of a positive $r$--cycle is the identity, whereas the $r^{th}$ power of a negative $r$--cycle is the product of $r$ transpositions  $(s_i \oo{s_i})$. Positive and negative $r$--cycles project to $r$--cycles under the natural surjection $B_n \to S_n$. For example, $(1 \, \oo3 \, 2)(\oo1 \, 3 \, \oo2)$ is a positive $3$-cycle and $(1 \, \oo3 \, 2 \, \oo1 \, 3 \, \oo2)$ is a negative $3$-cycle which both project to the $3$--cycle $(1\, 3 \, 2) \in S_n$. 

In 1930 Young \cite{YoungHyperoctahedral} proved that signed permutations factor uniquely as a product of positive and negative cycles, and two signed permutations are conjugate if and only if they have the same signed cycle type. We represent the signed cycle type of a signed permutation by a double partition $(\lambda, \mu)$, where $\lambda$ is a partition with a part of length $r$ for each positive $r$--cycle, and $\mu$ a partition encoding the negative $r$--cycles.

\subsubsection{ Character polynomials for $B_n$} \label{SubsectionCharacterPolys}
Given a signed permutation $\sigma$, let $X_r(\sigma)$ denote the number of positive $r$--cycles in its signed cycle type, and let $Y_r(\sigma)$ be the number of negative $r$--cycles. Then $X_r$ and $Y_r$ define class functions on the disjoint union $\coprod_{n \geq 0} B_n$. A \emph{character polynomial} with coefficients in a ring $\k$ is a polynomial $P \in \k[X_1, Y_1, X_2, Y_2, \ldots]$, in analogy to the character polynomials for the symmetric groups defined by MacDonald \cite[I.7 Example 14]{MacdonaldSymmetric}. A character polynomial $P$ restricts to a class function on $B_n$ for every $n \geq 0$; we denote its restriction by $P_n$. We define the \emph{degree} of a character  polynomial by setting $\deg(X_r)=\deg(Y_r)=r$ for all $r \geq 1$.

Let $\mu$ and $\lambda$ be partitions, and let $n_r$ be the function on the set of partitions that takes a partition and outputs the number of parts of size $r$. The space of character polynomials is spanned by polynomials of the form 	$$P_{\mu, \y} = { X \choose \mu }{ Y \choose \y} : = \prod_{r \geq 1} { X_r \choose n_r(\mu) } { Y_r \choose n_r(\y) }. $$ When $n=|\lambda|+|\mu|$ the restriction of $P_{\mu, \y}$ to $B_n$ is the indicator function for signed permutations of signed cycle type $(\lambda, \mu)$.

\subsubsection{ Classification of irreducible $B_n$--representations} The irreducible complex representations of $B_n$ are in natural bijection with \emph{double partitions} of $n$, that is, ordered pairs of partitions $(\lambda, \mu)$ such that $|\lambda|+|\mu|=n$. These irreducible representations are constructed from representations of the symmetric group $S_n$; this construction is described (for example) in Geck--Pfeiffer \cite{GeckPfeiffer}.

From the canonical surjection $B_n \to S_n$ we can pull back representations of $S_n$ to $B_n$. Let $V_{(\y, \varnothing)}$ denote the irreducible $B_n$--representation pulled back from the $S_n$--representation associated to the partition $\y$ of $n$. Let $V_{(\varnothing, (n))}$ denote the $1$-dimensional representation of $B_n$ where a signed permutation acts by $1$ or $-1$ depending on whether it reverses an even or odd number of signs. Then for any partition $\mu$ of $n$ we denote $V_{(\varnothing, \mu)} := V_{(\mu, \varnothing)} \otimes_{\C} V_{(\varnothing, (n))}$. In general, for partitions $\lambda$ of $k$ and $\mu$ of $m$, we define the $B_{k+m}$--representation
	$$V_{(\lambda, \mu)} := \Ind_{B_k \times B_m}^{B_{k+m}} V_{(\lambda, \varnothing)} \boxtimes V_{( \varnothing, \mu)}.$$
For each double partition $(\lambda, \mu)$ the representation $V_{(\lambda, \mu)}$ is distinct and irreducible, and this construction gives a complete list of irreducible complex $B_n$--representations. As with the symmetric group, each complex irreducible representation is defined over the rational numbers, and each representation is self-dual \cite[Corollary
3.2.14]{GeckPfeiffer}.





\subsection{The theory of $\FIW$--modules} \label{SubsectionBackgroundFIW}

Church--Ellenberg--Farb \cite{CEF} introduced the theory of \emph{FI--modules} to study sequences of representations of the symmetric groups $S_n$. FI denotes the category of finite sets and injective maps; an \emph{$\FI$--module} over a ring $R$ is a functor from FI to the category of $R$--modules.  Their results were generalized by the second author to sequences of representations of the classical Weyl groups in type B/C and D \cite{FIW1, FIW2}.

\begin{defn}{\bf(The categories FI and $\FI_{BC}$).} Following Church--Ellenberg--Farb \cite{CEF}, we let FI denote the category of finite sets and injective maps. We write $\bn$ or $[n]$ to denote the object $\{1,2,\ldots, n\}$ and ${\bf 0} := \varnothing$. Following work of the second author \cite{FIW1, FIW2}, we let $\FI_{BC}$ denote the type B/C analogue of FI; we may define $\FIBC$ to be the category where the objects are finite sets $\bn := \{1, \oo1, \ldots, n, \oo{n}\}$, and the morphisms are all injective maps $f: \bm \to \bn$ satisfying $f(\oo{a})=\oo{f(a)}$ for all $a$ in $\bm$. Notably the endomorphisms $\End(\bn)$  of FI are the symmetric groups $S_n$, and the endomorphisms $\End(\bn)$  of $\FI_{BC}$ are the hyperoctahedral groups $B_n$. We denote the families of Weyl groups $S_n$ and $B_n$ generically by $\W_n$, and we denote the categories FI and $\FI_{BC}$ generically by $\FIW$. \end{defn}
The following definitions appear in Church--Ellenberg--Farb \cite{CEF} and Wilson \cite{FIW1}.

\begin{defn}{\bf((Graded) $\FIW$--modules and $\FIW$--algebras).}  Fix $\W_n$ to denote either $S_n$ or $B_n$. Let $\k$ be a ring, assumed commutative and with unit.  An \emph{$\FIW$--module} $V$ over $\k$ is a functor from $\FIW$ to the category of $\k$--modules; its image is sequence of $\W_n$--representations $V_n:=V(\bn)$ with actions of the $\FIW$ morphisms. A \emph{graded $\FI_{BC}$--module $V^*$} over a ring $\k$ is a functor from $\FI_{BC}$ to the category of graded $\k$--modules; a \emph{graded $\FIW$--algebra $A^*$} over $\k$ is a functor from $\FIW$ to the category of graded $\k$--algebras.  Each graded piece $V^d$ or $A^d$ inherits an $\FIW$--module structure. We will refer to $d$ as the \emph{graded--degree} and $n$ as the \emph{$\FIW$--degree} of the $\k$--module $V^d_n$.
\end{defn}

\begin{defn}{\bf(Finite generation; degree of generation; finite type).} Let $\W_n$ denoted either the family of symmetric groups or the family of signed permutation groups. An $\FIW$--module $V$ is \emph{generated (as an $\FIW$--module)} by elements $\{v_i\} \subseteq \coprod_n V_n$ if $V$ is the smallest $\FIW$--submodule containing those elements.  A graded $\FIW$--algebra $A^*$ is \emph{generated (as an $\FIW$--algebra)} by the set $\{v_i\} \subseteq \coprod_n A^*_n$ if $A$ is the smallest $\FIW$--subalgebra containing $\{v_i\}$. An $\FIW$--module or $\FIW$--algebra $V$ is \emph{finitely generated} if it has a finite generating set, and $V$ is \emph{generated in $\FIW$--degree $\leq m$} if it has a generating set $v_i \in V_{m_i}$ such that $m_i \leq m$ for all $i$.  A graded $\FIW$--module or algebra $V^*$ has \emph{finite type} if each graded piece $V^d$ is a finitely generated $\FIW$--module. 
	\end{defn}
	
	\begin{defn}{\bf(Weight; slope).}
		Church--Ellenberg--Farb defined an FI--module $V$ over a subfield $\k$ of $\C$ to have \emph{weight $\leq m$} if for every $n$, every irreducible $S_n$--representation $V_{\lambda}$, $\lambda=(\lambda_1, \lambda_2, \ldots, \lambda_r)$, occuring in $V_n$ satisfies $(n-\lambda_1) \leq m$.  Analogously an $\FI_{BC}$--module $V$ over a subfield of $\C$  has \emph{weight $\leq m$} if for all $n$ every irreudcible $B_n$--representation $V_{(\lambda, \mu)}$, $\lambda=(\lambda_1, \lambda_2, \ldots, \lambda_r)$, satisfies $(n-\lambda_1) \leq m$. 
A graded $\FIW$--module or algebra $V^*$ has \emph{finite slope $M$} if $V^d$ has weight at most $dM$.
	\end{defn}
		$\FI$-- and $\FI_{BC}$--modules over $\k$ generated in degree at most $m$ necessarily have weight at most $m$ \cite[Proposition 2.3.5]{CEF}, \cite[Theorem 4.4]{FIW1}. 

Some of the early results on FI-- and $\FI_{BC}$--modules concern the implications of finite generation for the structure of an $\FIW$--module and its characters: over characteristic zero, a finitely generated $\FIW$--module is \emph{uniformly representation stable} in the sense of Church--Farb \cite[Definition 2.3]{CF}, and its characters are \emph{eventually polynomial}.  The following result was proved by Church--Ellenberg--Farb \cite{CEF} in type A and the second author \cite{FIW1, FIW2} in types B/C and D. 

\begin{thm} \label{REPSTABILITY}{\bf (Constraints on finitely generated $\FIW$--modules).} \label{ThmConstraintsfgmodules} Let $V$ be an $\FIW$--module over a subfield $\k$ of $\C$ which is finitely generated in degree $\leq m$. 
	\begin{itemize}
\item {\bf (Uniform representation stability)}{\rm (\cite[Theorem 1.13]{CEF},  \cite[Theorem 4.27]{FIW1}).} The sequence $V_n$ is uniformly representation stable with respect to the maps induced by the $\FIW$--morphisms the natural inclusions $I_n: \bn \to ({\bf n+1})$, stabilizing once $n \geq m + stabdeg(V)$. In particular, in the decomposition of $V_n$ into irreducible $\W_n$--representations, the multiplicities of the irreducible representations are eventually independent of $n$ in the sense of representation stability \cite[Definition 2.3]{CF}.  
\item 
{\bf (Character polynomials)}{\rm (\cite[Proposition 3.3.3 and Theorem 3.3.4]{CEF},   \cite[Theorem 4.6]{FIW2}).}  Let $\chi_n$ denote the character of the $\W_n$--representation $V_n$. Then there exists a unique character	polynomial $F_V$ of degree at most $m$ such that $F_V (\s) = \chi_n(\s)$ for all $\s \in \W_n$, for all $n$ sufficiently large.
\end{itemize}
\end{thm}

A major tool in the analysis of finitely generated FI-- and $\FI_{BC}$--modules is their relationship to represented functors and their subfunctors. 

\begin{defn} {\bf (The $\FIW$--modules $M(\bm)$).} \label{DefnM(m)} We denote the \emph{represented} $\FI$ or $\FI_{BC}$--module over $\k$ by $$ M(\bm):= \k\big[\Hom_{\FI}(\bm, -)\big] \qquad \qquad M_{BC}(\bm):= \k\big[\Hom_{\FI_{BC}}(\bm, -)\big] $$

	The $\FI$--module $M(\bm)_n \cong \k\big[ \Hom_{\FI}(\bm, \bn)\big]$ has $\k$--basis $$S_n / S_{n-m} \cong \{ e_{i_1, i_2, \ldots, i_m } \; | \; (i_1, i_2, \ldots, i_m) \text{ is an ordered $m$--tuple of distinct positive elements of $\bn$} \}$$ where the $m$--tuple $(i_1, i_2, \ldots, i_m)$ encodes the image of the $m$ elements $(1, 2, \ldots m)$ under an $\FI$--morphism $\bm \to \bn$. 
Similarly, the $\FI_{BC}$--module $M_{BC}(\bm)_n \cong \k\big[ \Hom_{\FI_{BC}}(\bm, \bn)\big]$ has $\k$--basis $$B_n / B_{n-m} \cong \left\{ e_{i_1, i_2, \ldots, i_m } \; \middle| \begin{array}{l} \; (i_1, i_2, \ldots, i_m) \text{ is an ordered $m$--tuple of elements of $\bn$;} \\ \text{ at most one of $a$ or $\oo{a}$ appears at most once} \end{array} \right\}$$ where the $m$--tuple $(i_1, i_2, \ldots, i_m)$ encodes the image of the $m$ elements $(1, 2, \ldots m)$ under an $\FI_{BC}$--morphism $\bm \to \bn$.  By an orbit--stabilizer argument we have isomorphisms of $\W_n$--representations $$ M(\bm)_n \cong \k \big[S_n/S_{n-m} \big] \qquad \qquad M_{BC}(\bm)_n \cong \k \big[B_n /B_{n-m}\big]. $$

	We denote the functors  $M(\bm)$ and $ M_{BC}(\bm)$ generically  by $M_\W(\bm)$. An $\FIW$--module $V$ over $\k$ is finitely generated in degree $\leq p$ if and only if it is a quotient of a finite direct sum of represented functors $M_\W(\bm)$ with $m \leq p$; see Church--Ellenberg--Farb \cite[Proposition 2.3.5]{CEF} and Wilson \cite[Proposition 3.15]{FIW1}. 
\end{defn}

\subsection{Asymptotics and asymptotic notation} \label{SubsectionBackgroundAsymptotics}

The following terminology features in the results of Section \ref{ASYMSection} on asymptotic stability.

\begin{defn}{\bf (Asymptotic equivalence; asymptotic bounds; Big and little O notation).} For functions $f(d)$ and $g(d)$, $d \in \Z_{\geq 0}$, we say that $f$ is \emph{asymptotically equivalent} to $g$ and write $f \sim g$  if $$\lim_{d \to \infty} \frac{f(d)}{g(d)}=1.$$



\noindent We say $f$ is order $o(g)$ if $$\lim_{d \to \infty} \frac{f(d)}{g(d)}=0.$$ We say that $f$ is order $O(g)$ or that $f$ is \emph{asymptotically bounded} by $g$ if $$ |f(d)| \leq C|g(d)| \qquad \text{for some constant $C$ and all $d$ sufficiently large}. $$ Note that the set of functions in $O(g)$ are closed under linear combinations (though not products).  
\end{defn}

\begin{defn} {\bf (Subexponential growth).} We say that a function $f(d)$ is \emph{subexponential} if $f$ is order $2^{o(d)}$. Notably, if $f$ is subexponential in $d$ and $0\leq r <1$ then the series $\sum_d f(d)r^d$ converges absolutely.  
\end{defn}

A classical example:
\begin{prop} [Hardy--Ramanujan \cite{HardyPartitions}, Uspensky \cite{UspenskyPartitions}]{\bf (Growth of partitions)}. \label{PartitionsSubexponential}  The number $\mathcal P(n)$ of partitons  of $n$ satisfies $$\mathcal P(n) \sim \frac {1} {4n\sqrt3} e^{\pi \sqrt {\frac{2n}{3}}}.$$
In particular $\mathcal P(n)$ is subexponential in $n$. 
\end{prop}

\section{Asymptotic stability}\label{ASYMSection}

The  main goal of this section is to prove  Theorem \ref{LIMIT}, a  general convergence result for $\FIW$-algebras finitely generated in $\FIW$-degrees $0$ and $1$.  Using the terminology and notation from the previous section, Theorem B can be restated as  Theorem \ref{LIMIT} below. We collect first  some preliminary combinatorial results in Section \ref{SectionCombinatorialResults}  to study the asymptotic behavior of $\FIW$--modules and their character polynomials. We then use these results in Section \ref{PROOFMAIN} to obtain  Theorem \ref{LIMIT}.   

\subsection{Asymptotics of character polynomials} \label{SectionCombinatorialResults}

Each of the results Propositions \ref{STABPOL}, Lemma \ref{CHARALGSTAB}, and Lemma \ref{EQUIV} mirrors a result proven by Church--Ellenberg--Farb \cite{CEFPointCounting} in Type A, and their methods can generally be modified to give arguments in Type B/C. For the sake of completeness we briefly describe these proofs in the case of the hyperoctahedral groups.

\begin{prop} \label{STABPOL} {\bf (Stability for inner products of character polynomials).} Fix a family $\W_n$ to be $S_n$ or $B_n$. Let $\k$ be a subfield of $\C$. Let $P, Q$ be two $\W_n$ character polynomials. Then the inner product $\langle Q_n, P_n \rangle_{\W_n}$  is independent of $n$ for $n \geq \deg(P) + \deg(Q)$. 
\end{prop}

\begin{proof} This result was proved for $S_n$ by Church--Ellenberg--Farb \cite[Proposition 3.9]{CEFPointCounting} and their proof adapts readily to type B/C; we briefly summarize this proof. Since  $\langle Q_n, P_n \rangle_{B_n} = \langle 1, Q_nP_n \rangle_{B_n}$ it suffices to check in the case when $Q=1$ and $P$ is the element
	$$P_{\mu, \y} = { X \choose \mu }{ Y \choose \y} : = \prod_{r \geq 1} { X_r \choose n_r(\mu) } { Y_r \choose n_r(\y) } $$
	associated to two partitions $\mu$ and $\y$. 
	Let $d = |\mu|+|\y|$, and notice that $\deg(P)=d$. Let  $\delta_{\y, \mu}$ denote the indicator function for the conjugacy class $(\mu, \y)$ in $B_d $. Let $N_{\mu, \y}$ denote the size of the conjugacy class $(\mu, \y)$ in $B_d$. Then for $n \geq d$, 
	\begin{align*}
	\left\langle { X \choose \mu }{ Y \choose \y}, 1 \right\rangle_{B_n} & 
	=  \frac{1}{|B_n|} \sum_{ \s\in B_n} \sum_{S \subseteq [n], |S|=d} \delta_{\y, \mu}(\s \vert_{S \cup \oo{S}}) \end{align*} where $\delta_{\y, \mu}(\s \vert_{S \cup \oo{S}}) :=0$ if ${S \cup \oo{S}}$ is not stabilized by $\sigma$. So \begin{align*}
	\left\langle { X \choose \mu }{ Y \choose \y}, 1 \right\rangle_{B_n} &  = \frac{1}{|B_n|} {n \choose d} N_{\mu, \y} |B_{n-d}|   \\
	& = \frac{N_{\mu, \y}}{2^d\; d!} 
	\end{align*}
Hence for $n \geq d= \deg(P)+\deg(Q)$ this inner product is independent of $n$, as claimed. 
\end{proof}

\begin{lem} \label{CHARALGSTAB} {\bf (A convergence result for $\FIW$--algebras of finite type and slope).} Let $\W_n$ denote one of the families $S_n$ or $B_n$. Let $\k$ be a subfield of $\C$. 
	Suppose that $A^*$ is a $\FIW$--algebra over $\k$ of finite type and slope $M$. Then for each $d$ and any $\W_n$ character polynomial $P$, the following limit exists: $$  \lim_{n \to \infty} \langle P_n, A^d_n \rangle_{\W_n}.$$  \end{lem}

\begin{proof}
	This result is proved in Type A by Church--Ellenberg--Farb; see the paragraph before \cite[Corollary 3.11]{CEFPointCounting}. Their argument carries over to Type B/C as follows. By assumption, for each fixed $d$ the $\FI_{BC}$--module $A^d_n$ is finitely generated and has weight at most $dM$. The second author proved that therefore for some $D_d \in \Z$, the characters of the sequence $\{A^d_n\}_n$ are given by a character polynomial of degree at most $dM$ for all $n \geq D_d$ \cite[Theorem 4.16]{FIW2}; see Theorem \ref{ThmConstraintsfgmodules}. By Proposition \ref{STABPOL}, then, the value $\langle P_n, A^d_n \rangle_{B_n}$ is independent of $n$ once $n= \max\{D_d, dM+\deg(P) \}$. This stable value gives the limit 
	$$  \lim_{n \to \infty} \langle P_n, A^d_n \rangle_{B_n} = \langle P_N, A^d_N \rangle_{B_N}, \qquad N= \max\{D_d, dM+\deg(P) \}. \qedhere$$ 
\end{proof}

\begin{lem}
\label{FGCHARALGSTAB}  {\bf (A convergence result for finitely generated $\FIW$--algebras).} Let $\k$ be a subfield of $\C$, and let $\W_n$ represent one of the families $S_n$ or $B_n$. Suppose that $A^*$ is an associative $\FIW$--algebra over $\k$ that is generated as an $\FIW$--algebra by finitely many elements of positive graded-degree. Then for each $d$ and any $\W_n$ character polynomial $P$, the following limit exists: $$  \lim_{n \to \infty} \langle P_n, A^d_n \rangle_{\W_n}.$$
\end{lem}

\begin{proof}
We first prove the result in type B/C.  Let $V$ be the graded $\FI_{BC}$--module additively generated by the generating set for the graded $\FI_{BC}$--algebra $A^*$. By assumption on the generators $V$ will be supported in finitely many positive degrees. By construction $V$ has finite type, and by Wilson \cite[Theorem 4.4]{FIW1}, $V$ has finite slope. It follows that $A^*$ is a quotient of the free associative $\FI_{BC}$--algebra on $V$, $$k\langle V \rangle := \bigoplus_{j=0}^{\infty} V^{\otimes j};$$ see \cite[Definition 5.9]{FIW1}. By \cite[Proposition 5.2]{FIW1}, tensor products respect finite generation, and degree of generation is additive. Hence if we let $M$ be the largest $\FI_{BC}$--degree of the generators of $A^*$, the $\FI_{BC}$--module $A^d$ is finitely generated in degree $\leq dM$.  Thus $A^*$ has finite type and finite slope $M$; see \cite[Proposition 5.10]{FIW1}.  By Lemma \ref{CHARALGSTAB}, the limit $\lim_{n \to \infty} \langle P, A^k_n \rangle_{B_n}$ exists. 

	The proof for FI follows the same argument.  The graded $\FI$--algebra $A^*$ has finite type and slope by Church--Ellenberg--Farb \cite[Proposition 3.2.5 and Theorem 4.2.3]{CEF}.
\end{proof}

The following result will be used in the proof of Theorem \ref{LIMIT}.

\begin{lem} \label{EQUIV} {\bf (On bounding growth of the graded pieces of $\FIW$--algebras).} Let $A^d$ be the $d^{th}$ graded piece of a graded $\FIW$--module over a subfield $\k$ of $\C$. Let $g(d)$ be a function. The following are equivalent: 
	\begin{enumerate}
		\item[I.] For each $a \geq 0$ there is a function $F_a(d)$ that independent of $n$ and order $O(g)$ such that $$ \dim_{\k} \Big( (A^d_n)^{\W_{n-a}} \Big) \leq F_a(d) \qquad \text{for all $d$ and $n$}. $$
		\item[II.] For each  $\W_n$ character polynomial $P$ there is a function $F_P(d)$ that is independent of $n$ and order $O(g)$ such that $$ | \langle P_n,  A^d_n \rangle_{B_n} | \leq F_P(d) \qquad \text{for all $d$ and $n$}. $$
	\end{enumerate}	
\end{lem}

\begin{proof} A special case of this result is the equivalence of the two conditions in \cite[Definition 3.12] {CEFPointCounting} for $S_n$; although Church--Ellenberg--Farb only state the equivalence in Type A for the case that $g$ is subexponential, their proof implies the result in Type A for general functions $g$. Their arguments may be adapted to the hyperoctahedral groups, and we summarize the proof in type B/C below.

Let $\k$ denote the trivial $B_n$--representation. Assume (II) holds. By Frobenius reciprocity, 
	\begin{align*}
	\dim_{\k} ((A^d_n)^{B_{n-a}}) &= \langle \k, \Res^{B_n}_{B_{n-a}} A^d_n \rangle_{B_{n-a}} \\
	&=  \langle \Ind_{B_{n-a}}^{B_n} \k,   A^d_n \rangle_{B_{n}} \\ 
	&=  \langle   \k[B_n / B_{n-a}],   A^d_n \rangle_{B_{n}} \\ 
	&=  \langle 2^a a! {X_1 \choose a},   A^d_n \rangle_{B_{n}} \\
& \leq F_{2^a a! {X_1 \choose a}}(d) \qquad \text{for all $d$ and $n$}.
	\end{align*} 
	So (II) implies (I). Now assume (I) holds and consider any double partition $\lambda$ of $a$ and the associated $B_a$--representation $V_{\lambda}$ with character $\chi^\lambda$. The character of the induced representation  $$\Ind^{B_n}_{B_a \times B_{n-a}} V_{\y} \boxtimes \k$$ is equal (for any $n$) to the character polynomial $$ P^{\lambda} =  \sum_{ \substack{(\alpha, \beta) \\ |\alpha|+|\beta|=a}} \chi^{\lambda}(\alpha, \beta) \prod_r {X_r \choose n_r(\alpha)} \prod_r {Y_r \choose n_r(\beta)} := \sum_{ \substack{(\alpha, \beta) \\ |\alpha|+|\beta|=a}} \chi^{\lambda}(\alpha, \beta) {X \choose \alpha}{Y \choose \beta}. $$ Here $\chi^{\lambda}(\alpha, \beta)$ denotes the value of the character $\chi^{\lambda}$ on a signed permutation of signed cycle type $(\alpha, \beta)$. The character polynomials $P^{\lambda}$ are an additive basis for the space of hyperoctahedral character polynomials, and so it suffices to bound $ | \langle P^{\lambda}_n, A^d_n \rangle_{B_n} |$. Observe that
	\begin{align*}
	\dim_{\k}((A^d_n)^{B_{n-a}}) & \geq \langle V_{\lambda}, (A^d_n)^{B_{n-a}} \rangle_{B_{a}} \\ 
	&=  \langle  V_{\lambda} \boxtimes \k \;, \; \Res^{B_n}_{B_a \times B_{n-a}}A^d_n \rangle_{B_{a}\times B_{n-a}} \\ 
	&=  \langle  \Ind^{B_n}_{B_a \times B_{n-a}} V_{\y} \boxtimes \k \; , \;   A^d_n \rangle_{B_{n}} \\
	&=  \langle  P^{\lambda}_n \; , \;   A^d_n \rangle_{B_{n}} 
	\end{align*}
	which gives the desired bound.
\end{proof} 

\subsection{Convergence for subquotients of $\FIW$-algebras generated in low degree}\label{PROOFMAIN}

The main result of this section is Theorem \ref{LIMIT}, which states that graded submodules of $\FIW$--algebras finitely generated in  $\FIW$--degree at most $1$ satisfy our desired convergence result.

\begin{thm}\label{LIMIT} {\bf (Criteria for convergent $\FIW$--algebras).} Let $\W_n$ denote one of the families $S_n$ or $B_n$. 
  For nonnegative integers $b,c$, define a graded $\FIW$--module $V \cong M_{\W}({\bf 0})^{\oplus b} \oplus M_{\W}({\bf 1})^{\oplus c}$ over a subfield $\k$ of $\C$ with positive gradings. Let $\Gamma^*$ denote the free symmetric, exterior, or graded--commutative $\FIW$--algebra generated by $V$. Let $A^*$ be any $\FIW$--algebra subquotient of $\Gamma^*$.
 Then for any character polynomial $P$ and $q>1$ the following sum converges absolutely  
$$ \sum_{d=0}^{\infty}  \frac{ \lim_{n \to \infty} \langle P_n, A^d_n \rangle_{\W_n}}{q^d}. $$ 
More generally, this sum will converge absolutely for any collection of $\W_n$--representations $\{A^d_n \; | \; d, n \geq 0 \}$ such that $A^d_n$ is a $\W_n$--equivariant subquotient of $\Gamma^d_n$. 
	\end{thm}

\noindent We assume that the gradings on the algebra generators $M_{\W}({\bf 0})^{\oplus b} \oplus M_{\W}({\bf 1})^{\oplus c}$ are $\W_n$--invariant. 
	
	We will use Theorem \ref{LIMIT} to show convergence for coinvariant algebras in Proposition \ref{CONVCOIN}, and to obtain Theorem \ref{ASYM_TORI} on stability for statistics on maximal tori.

The proof of Theorem \ref{LIMIT} will use the following result on the asymptotics of enumerating coloured partitions. Proposition \ref{TNCSubexponential} generalizes the result that partitions of $n$ grow subexponentitally in $n$ (see Proposition \ref {PartitionsSubexponential}). 

\begin{prop}\label{TNCSubexponential} For integers $N$ and $C$, let $T(N,C)$ denote the number of ways that $N$ balls may be first each coloured by one of $C$ colours, and then partitioned into any number of multisets. Balls of the same colour are indistinguishable. For every fixed $C$ the sequence $T(N,C)$ grows subexponentially in $N$.
\end{prop}

For example, $T(3,2)=14$, corresponding to the 14 coloured partitions
{\small \begin{align*}   &\{\bullet, \bullet, \bullet \}, \; \;  \{\bullet, \bullet, \circ \}, \; \;  \{\bullet, \circ, \circ \}, \; \;  \{\circ,  \circ,  \circ \}, \; \;   \{\bullet, \bullet \} \cup \{ \bullet \}, \; \;  \{\bullet, \bullet \} \cup \{ \circ \}, \; \;  \{\bullet, \circ \} \cup \{ \bullet \}, \; \;  \{\bullet, \circ \} \cup \{ \circ \}, \; \;  \\
&  \{\circ, \circ \} \cup \{ \bullet \}, \; \;  \{\circ, \circ \} \cup \{ \circ \}, \;\; \{\bullet\} \cup \{ \bullet\} \cup \{ \bullet \}, \; \; \{\bullet\} \cup \{ \bullet\} \cup \{ \circ \}, \; \;  \{\bullet\} \cup \{ \circ\} \cup \{ \circ\}, \; \;  \{\circ\} \cup \{ \circ\} \cup \{ \circ \}.
\end{align*} }
\vspace{-1.5em}
\begin{proof} Suppose throughout that $C$ is fixed.  We will create a code with a codeword recording each partition of $N$ balls coloured with $C$ colours. To verify that $T(N,C)$ is subexponential in $N$, it suffices to check that the maximal length of these codewords is sublinear in $N$.  

Consider a partition $P=\{P_1, P_2, \ldots, P_{\ell}\}$ of $N$ coloured balls. We will create a distinct coding scheme for (i) the ``small" parts $P_i$ containing at most $s:=(\log N)^2$ balls, and (ii) the ``large" parts with more than $s$ balls. 

\begin{enumerate}
\item[(i)] (Small parts). There are ${ s+C \choose C }$ possible sets of at most $s$ balls coloured by $C$ colours (this is the equal to the number of monomials of $C$ variables of degree at most $s$). Each of these sets appears as a part $P_i$ in our coloured partition $P$ at most $N$ times. We can then encode these sets by an ordered ${ s+C \choose C }$--tuple of integers from 0 to $N$, each recording the number of times the corresponding set appears in $P$. The number of characters needed to encode this tuple is asymptotically bounded in $N$ by the function \begin{align*} &{ s+C \choose C } (\log N ) \\& < (s+C)^C (  \log N) \\& = ((\log N)^2+C)^C ( \log N ).\end{align*}  
\item[(ii)] (Large parts). There are fewer than $\frac{N}{s}$ parts $P_i$ in $P$ with cardinality strictly greater than $s$. For each of these parts, we record an ordered $C$--tuple of integers from $0$ to $N$ encoding the number of balls of each given colour in that part. The total number of characters needed to do this is asymptotically bounded by $$C\left(\frac{N}{s}\right)(\log N) = \frac{CN}{(\log N)} $$ 
\end{enumerate}

Combining (i) and (ii) we find that the maximal length of a codeword is asymptotically bounded by
\begin{align*} \Big((\log N)^2+C\Big)^C ( \log N)+\frac{CN}{\log N} .
\end{align*}
We conclude that $\log(T(N,C))$ grows sublinearly in $N$, and so $T(N,C)$ is subexponential in $N$ as claimed.
\end{proof}

Having established a bound on the growth of $T(N,C)$, we now need one final result in order to prove Theorem \ref{LIMIT}.

\begin{lem} \label{SUBEXP} Let $\W_n$ denote one of the families $S_n$ or $B_n$. 
 For nonnegative integers $b,c$, define a graded $\FIW$--module $V \cong M_{\W}({\bf 0})^{\oplus b} \oplus M_{\W}({\bf 1})^{\oplus c}$ over a subfield $\k$ of $\C$ with positive gradings. Let $\Gamma^*$ denote the free symmetric, exterior, or graded--commutative $\FIW$--algebra generated by $V$. Let $A^*$ be any $\FIW$--algebra subquotient of $\Gamma^*$. Then for each $a \geq 0$ there is a function $F_a(d)$ that is independent of $n$ and subexponential in $d$ so that  
 $$\dim_{\k}\big((A^d_n)^{B_{n-a}}\big) \leq F_a(d) \qquad \text{for all $n$ and $d$.} $$ 
  It follows that,  for any $\W_n$ character polynomial $P$, there exists a function $F_P(d)$  independent of $n$ and subexponential in $d$ such that 
		$$ |\langle P_n, A^d_n \rangle_{\W_n} | \leq F_P(d). $$ 
	More generally, if $\{ A^d_n \; \vert \; d, n \geq 0 \}$ is any collection of $\W_n$--representations such that $A^d_n$ is an $\W_n$--equivariant subquotient of $\Gamma^d_n$, then there exists a function $F_P(d)$ as above.

	\end{lem}
		\begin{proof} 
		The proof in Types A and B/C are extremely similar; we describe the proof in detail for the hyperoctahedral groups and then briefly outline how to adapt the proof to the symmetric groups. 
		By Lemma \ref{EQUIV} it suffices to show the first statement:  for each $a \geq 0$ there is a function $F_a(d)$ that is independent of $n$ and subexponential in $d$ so that  $$\dim_{\k}\big((A^d_n)^{B_{n-a}}\big) \leq F_a(d) \qquad \text{for all $n$ and $d$.} $$  		
		
		Fix $a \geq 0$. By assumption $A^d_n$ is a $\W_n$--subquotient of the $\W_n$--representation $\Gamma^d_n$. Since taking $B_{n-a}$--invariants is exact in characteristic zero, the graded algebra $(A_n^*)^{B_{n-a}}$ is a subquotient of $(\Gamma_n^*)^{B_{n-a}}$.  It suffices then to find a function $F_a(d)$ that is independent of $n$ and subexponential in $d$ so that $$\dim_{\k}\big((\Gamma^d_n)^{B_{n-a}}\big) \leq F_a(d) \qquad \text{for all $n$ and $d$.} $$

The monomials of $\Gamma^d_n$ are (graded-- or anti--) commutative words in the $(b+2nc)$ variables $$y^{(1)}, y^{(2)}, \ldots, y^{(b)} \quad \in M_{BC}({\bf 0})^{\oplus b}_n \qquad \text{and}$$ $$ x_1^{(1)}, x_{\oo1}^{(1)}, x_2^{(1)}, x_{\oo2}^{(1)}, \ldots, x_n^{(1)}, x_{\oo{n}}^{(1)}, \ldots \ldots, x_1^{(c)}, x_{\oo1}^{(c)}, \ldots, x_{n}^{(c)} x_{\oo{n}}^{(c)} \quad \in M_{BC}({\bf 1})^{\oplus c}_n.$$
By assumption each variable has positive graded-degree, so the degree-$d$ monomials must have length at most $d$.
The subgroup $B_{n-a}$ acts on these monomials diagonally by signed permutations on the subscript indices $$\{ a+1, \oo{a+1}, \ldots, n, \oo{n} \}.$$ Superscripts are fixed. To bound the number of $B_{n-a}$ orbits of monomials, we will in fact bound the (larger) number of $S_{n-a}$ orbits; $S_{n-a}$ acts by permuting the set $\{ a+1, \ldots, n \}$ and simultaneously permuting the set $\{ \oo{a+1}, \ldots,\oo{n} \}$. We can classify $S_{n-a}$--orbits of these monomials as follows: Each monomial contains a (possibly empty) subword in the $(b+2ac)$ variables $$y^{(1)}, y^{(2)}, \ldots, y^{(b)} \quad \in M_{BC}({\bf 0})^{\oplus b}_n \qquad \text{and}$$ $$ x_1^{(1)}, x_{\oo1}^{(1)}, x_2^{(1)}, x_{\oo2}^{(1)}, \ldots, x_a^{(1)}, x_{\oo{a}}^{(1)}, \ldots \ldots, x_1^{(c)}, x_{\oo1}^{(c)}, \ldots, x_{a}^{(c)} x_{\oo{a}}^{(c)} \quad \in M_{BC}({\bf 1})^{\oplus c}_n$$ There are $ {(b+2ac)+d \choose (b+2ac)}$ monomials of length at most $d$ in the first $(b+2ac)$ variables. For fixed $a,b,c$ this formula grows polynomially in $d$.

Now, consider each of the remaining variables $x_{i}^{(j)}$ to be 'coloured' by one of the $(2c)$ symbols $x^{(j)}_+$ or $x^{(j)}_-$, where the sign represents whether $i$ is positive or negative, for $j=1, \ldots, c$.  Next we partition these coloured variables into multisets with one set for each subscript $i$. The action of $S_{n-a}$ preserves this decomposition. 

Concretely, for example, if $a=2$ then the $S_{n-a}$--orbit of the monomial
$$\left(y^{(2)}\right)^3y^{(4)}y^{(5)}\left(x_{\oo1}^{(1)}\right)^2x_2^{(3)}x_{\oo2}^{(4)}
x_{3}^{(2)}\left(x_{3}^{(4)}\right)^2 x_{\oo3}^{(3)} x_{\oo3}^{(4)} \left(x_{4}^{(1)}\right)^3x_{5}^{(5)}\left(x_{\oo5}^{(3)}\right)^2 x_{\oo6}^{(2)} x_{8}^{(1)} x_{8}^{(3)}  $$
would be represented by the monomial $\left(y^{(2)}\right)^3y^{(4)}y^{(5)}\left(x_{\oo1}^{(1)}\right)^2 x_2^{(3)}x_{\oo2}^{(4)}$ and the coloured partition $\{ x^{(2)}_+, x^{(4)}_+, x^{(4)}_+, x^{(3)}_-, x^{(4)}_- \}\cup \{x^{(1)}_+,x^{(1)}_+,x^{(1)}_+ \}\cup \{x^{(5)}_+, x^{(3)}_-,x^{(3)}_-\} \cup \{x^{(2)}_-\} \cup\{ x^{(1)}_+,x^{(3)}_+\}.  $

 Given monomials of $N$ variables, the number of coloured partitions we may obtain in this fashion is, in the notation of Proposition \ref{TNCSubexponential}, $T(N,2c)$, which we proved in Proposition \ref{TNCSubexponential} grows subexponentially in $N$ for each fixed $c$. 

To get an upper bound for the number of $B_{n-a}$--orbits for the monomials in $\Gamma^*$ of degree $d$, we will bound the number of $B_{n-a}$--orbit of monomials in at most $d$ variables. Some of these monomials will have degree less than $d$. If some variables anticommute, then some monomials will be zero in $\Gamma^d$ and some monomials will vanish on passing to coinvariants, so we obtain a strict overcount. 
\begin{align*}
\text{\#} & \text{ of $B_{n-a}$--orbits of monomials in at most $d$ variables} \\ =& \sum_{\substack{j,k \\ j+k \leq d}} \left( \begin{array}{c} \text{\# degree $j$ monomials}\\ \text{in } y^{(1)}, \ldots, x^{(c)}_{\oo{a}} \end{array}\right)\left( \begin{array}{c} \text{\# $B_{n-a}$-orbits of degree $k$ }\\ \text{monomials in remaining variables }\end{array} \right) \\
\leq& \left[ \sum_{\substack{j,k \\ j+k \leq d}} \left( \begin{array}{c} \text{\# degree $j$ monomials}\\ \text{in }  y^{(1)}, \ldots, x^{(c)}_{\oo{a}} \end{array}\right) \right]\left( \begin{array}{c} \text{\# $B_{n-a}$-orbits of degree $d$ }\\ \text{monomials in remaining variables }\end{array} \right) \\
\leq&   {(b+2ac)+d \choose (b+2ac)}T(d,2c)
\end{align*}
We thus obtain a bound that (for constant $a,b,c$) is independent of $n$ and subexponential in $d$. 

The proof in the case of the symmetric group is similar and slightly simpler: our $S_{n-a}$ orbits may be represented  by a submonomial in the $(b+ac)$ variables $$y^{(1)}, y^{(2)}, \ldots, y^{(b)} \quad \in M({\bf 0})^{\oplus b}_n \qquad \text{and}$$ $$ x_1^{(1)}, x_2^{(1)},  \ldots, x_a^{(1)},  \ldots \ldots, x_1^{(c)}, x_2^{(c)} \ldots, x_{a}^{(c)}  \quad \in M({\bf 1})^{\oplus c}_n$$ and a partition that is coloured by the $c$ colours $j=1, 2, \ldots c$. We obtain an asymptotic bound 
$$   {(b+ac)+d \choose (b+ac)}T(d,c)$$ that is independent of $n$ and subexponential in $d$. 		
	\end{proof}

\begin{proof}[Proof of Theorem \ref{LIMIT}]  By Lemma \ref{FGCHARALGSTAB}, the  limit in the numerator $\lim_{n \to \infty} \langle P_n, A^d_n \rangle_{\W_n}$ exists for every $d$ and moreover equals $\langle P_N, A^d_N \rangle_{\W_N}$ for $N$ sufficiently large. By Lemma \ref{SUBEXP}, there exist a  function $F_P(d)$   subexponential in $d$ such that 
		$$ |\lim_{n \to \infty} \langle P_n, A^d_n \rangle_{\W_n} | \leq F_P(d). $$ 
		It follows that the sum $$ \sum_{d=0}  \frac{\lim_{n \to \infty}\langle P_n, A^d_n \rangle_{\W_n}}{q^d} $$
		converges absolutely.
\end{proof}

\begin{rem} \label{RemCoFI} {\bf (Convergence for graded $\FIW^{\text{op}}$--algebras generated in low degree).} Recall the $\FIW$--module $V \cong M_{\W}({\bf 0})^{\oplus b} \oplus M_{\W}({\bf 1})^{\oplus c}$ of Theorem $\ref{LIMIT}$. The module $V$ has a natural $\FIW^\text{op}$--module structure   defined (in the notation of Definition \ref{DefnM(m)}) by the maps 
\begin{align*}
 M_{\W}({\bf 1})_{n+1} &\longrightarrow   M_{\W}({\bf 1})_{n}\\
   e_i &\longmapsto  \left\{ \begin{array}{ll} e_i, & |i|=1, \ldots, n
 \\ 0, & |i| = n+1.  \end{array} \right. 
\end{align*}
and the isomorphisms $ M_{\W}({\bf 0})_{n+1} \cong  M_{\W}({\bf 0})_{n}$. (See Church--Ellenberg--Farb \cite[Section 4.1]{CEF} and Wilson \cite[Section 3]{FIW2} for details on the structure of simultaneous $\FIW$ and $\FIW^{\text{op}}$--modules.) These maps endow $\Gamma^*$  with the structure of an $\FIW^\text{op}$--algebra. Since $\W_n$--representations are self-dual, the representations $\Gamma^*_n$ are isomorphic whether viewed as representations of End$_{\FIW}({\bf n}) \cong \W_n$ or End$_{\FIW^{\text{op}}}({\bf n}) \cong \W_n$. If $A$ is any $\FIW^\text{op}$--algebra subquotient of $\Gamma^*$, then in particular the bigraded piece $A^d_n$ is a subquotient of $\Gamma^d_n$, and Theorem \ref{LIMIT} applies: for any character polynomial $P$ and $q>1$ the following sum converges absolutely  
$$ \sum_{d=0}^{\infty}  \frac{ \lim_{n \to \infty} \langle P_n, A^d_n \rangle_{\W_n}}{q^d}. $$ 
\end{rem}

\begin{rem} \label{RemarkFarbWolfson}
After circulating a preprint of this paper we discovered that in \cite[Theorem 3.2 and Corollary 3.3]{FarbWolfson}, Farb and Wolfson prove a result which is closely related to our Theorem \ref{LIMIT}. 

\begin{thm}[{\cite[Theorem 3.2 and Corollary 3.3]{FarbWolfson}}] \label{TheoremFarbWolfson}
Let $X$ be a connected space such that $\dim H^*(X; \Q) < \infty$. Then there exist constants $K, L > 0$ so that for
each $i \geq 0$, and for all $n \geq 1$, the Betti numbers of the $n$-fold symmetric product of $X$ are bounded subexponentially:
$$\dim H^i(\mathrm{Sym}^n(X); \Q) < Ke^{L \sqrt{i}}.$$
For each $0 \leq a \leq n$ exist constants $K, L > 0$ so that for
each $i \geq 0$, and for all $n \geq 1$, the coinvariants of the cohomology of $X^n$ are bounded subexponentially: 
$$\dim \Big( H^i(X^n; \Q)^{S_{n-a}}\Big) < Ke^{L \sqrt{i}}.$$
\end{thm} 

Farb--Wolfson use methods that are quite different from those of Theorem \ref{LIMIT}: their proof uses a result of Macdonald on the Poincar\'e polynomial of symmetric products \cite{MacDonaldSymmetricProduct}, and an analysis of this generating functions drawing on complex analysis. 

It should be possible to deduce the subexponential growth of the Betti numbers and their coinvariants in Theorem \ref{TheoremFarbWolfson} from Theorem \ref{LIMIT} and its proof; it would take additional work to conclude their precise asymptotic formulas. Conversely, by applying Theorem \ref{TheoremFarbWolfson} with appropriately chosen spaces $X$ and using the K\"unneth Formula, it should be possible to deduce Theorem \ref{LIMIT} in the case of graded-commutative algebras in type A. With some additional work it should be possible to leverage these results  to give an alternate route to proving Theorem \ref{LIMIT} in all cases. 
\end{rem}

\section{Point counts for maximal tori in type B and C}\label{MAXTORI} 

In this section we obtain the type B and C analogues of \cite[Theorem 5.6]{CEF}, a convergence result for certain statistics on tori in $\GL_n(\overline{\F_q})$. We combine a theorem of Lehrer with representation stability results of the second author and our discussion in Section \ref{ASYMSection} to compute asymptotic counts for $\Fr_q$-stable maximal tori for the symplectic and special orthogonal groups.

\subsection{Maximal tori in types B and C} \label{SectionMaxToriBackground}
Let $\k$ be an algebraically closed field. We denote by $\G$ one of the following two semisimple connected linear algebraic groups defined over $\k$: the symplectic group $\Sp_{2n}$ (type C) or the special orthogonal group $\SO_{2n+1}$ (type B). In the discussion below we identify the algebraic group $\G$ with its set of $\k$-rational points $\G(\k)$.  Recall that the symplectic group  is the algebraic group defined over $\k$ by
$$\Sp_{2n}(\k)=\{A\in \GL_{2n}(\k): A^TJA=J\},$$
where $J= \left[ {\begin{array}{cc} 0 & I_n \\ -I_n & 0 \\ \end{array} } \right]$ is the matrix associated to the standard symplectic form $$\omega(\vec{u},\vec{v})=u_1v_{n+1}+\ldots +u_nv_{2n}-u_{n+1}v_1-\ldots -u_{2n}v_n,\text{\  \ \ for\ \ \  } \vec{u},\vec{v}\in \k^{2n}.$$  The special orthogonal group  is defined as
$$\SO_{2n+1}(\k)=\{A\in \SL_{2n+1}(\k): A^TQA=Q\},$$
where $Q= \left[ {\begin{array}{ccc}1& 0 & 0\\ 0& 0 & I_n \\ 0&I_n &0 \\ \end{array} } \right]$ is the matrix associated to the quadratic form of Witt index $n$  $$q(u_0,u_1,\ldots,u_{2n})=u_0^2+u_1u_{n+1}+\ldots+u_nu_{2n},\text{ \ \ \ for\ \  \ \ }(u_0,u_1,\ldots,u_{2n})\in \k^{2n+1}.$$

A subgroup of $\G$ is called a {\it torus} if it is $\k$-isomorphic to a product of multiplicative groups $\mathbb{G}_m:=\GL_1$. 
For instance, we define standard maximal tori $T_0$ to be the diagonal subgroups $$T_0:=\big\{\text{diag}(\lambda_1,\ldots,\lambda_n,\lambda_1^{-1},\ldots,\lambda_n^{-1}): \lambda_i\in \k^\times\big\} \subseteq \Sp_{2n}(\k)$$ and $$T_0:=\big\{\text{diag}(1,\lambda_1,\ldots,\lambda_n,\lambda_1^{-1},\ldots,\lambda_n^{-1}): \lambda_i\in \k^\times\big\} \subseteq \SO_{2n+1}(\k).$$ In both cases, the maximal torus $T_0$ has dimension $n$.  

Let $\mathcal{T}$ denote the set of maximal tori of $\G$. Since all maximal tori in $\G$ are conjugate (\cite[Proposition 1.1]{SRI}), the action of $\G$ on $\mathcal{T}$ by conjugation is transitive and $\mathcal{T}\cong \G/N(T_0),$ where $N(T_0)$ is the normalizer of the torus $T_0$. In type B and C, the Weyl group $W(T_0):=N(T_0)/T_0$ is isomorphic to the hyperoctahedral group $B_n$. The group $B_n$ acts on matrices in $T_0$ by conjugation by permuting the $n$ eigenvalues $\lambda_1,\ldots,\lambda_n$ and transposing the inverse pairs $\lambda_i $ and $ \lambda_i^{-1}$.

\subsubsection{The action of Frobenius} \label{SubsectionFrobAction}

Let $q$ be an integral power of a prime $p$ and let $\k=\overline{\F_p}$. The standard Frobenius morphism $\Fr_q$  acts on a matrix $(x_{ij}) \in \G= \G(\overline{\F_q})$ by $\Fr_q:(x_{ij})\mapsto (x_{ij}^q)$. The set of fixed points $\G^{\Fr_q}:=\{g\in\G: \Fr_q(g)=g\}$ corresponds to the $\F_q$-points $\G(\F_q)$ of $\G$: the finite groups $\Sp_{2n}(\F_q)$ and $\SO_{2n+1}(\F_q)$.  
A maximal torus of $\G^{\Fr_q}$ is a  subgroup of $\G(\F_q)$ of the form $T^{\Fr_q}=\{g\in T: \Fr_q(g)=g\}$  for some $\Fr_q$-stable maximal torus $T$ of $\G$. In particular, since $T_0$ is $\Fr_q$-stable,

$$T_0^{\Fr_q}=\big\{\text{diag}(\lambda_1,\ldots,\lambda_n,\lambda_1^{-1},\ldots,\lambda_n^{-1}): \lambda_i\in \F_q^\times\big\}\text{ is a maximal torus of } \Sp_{2n}(\F_q)$$
and 
$$T_0^{\Fr_q}=\big\{\text{diag}(1,\lambda_1,\ldots,\lambda_n,\lambda_1^{-1},\ldots,\lambda_n^{-1}): \lambda_i\in \F_q^\times\big\}\text{ is a maximal torus of }\SO_{2n+1}(\F_q).$$
Other examples of $\Fr_q$-stable maximal tori in $\Sp_2(\overline{\F_p})$ are given in Example \ref{ToriExample} below. 

An  $\Fr_q$-stable torus  of $\G$  is defined over $\F_q$. We say that such a torus $T$ {\it splits over $\F_q$} if $T$ is $\F_q$-isomorphic to a product of multiplicative groups $\mathbb{G}_m$.  There is always a finite Galois extension of $\F_q$ over which a  given torus becomes diagonalizable,
hence a split torus. The  maximal torus $T_0$  above splits over $\F_q$ (or any field) and  the group $\G$ is said to be a {\it split} algebraic group. We will investigate statistics on the set $\mathcal{T}^{\Fr_q}$ of $\Fr_q$-stable maximal tori of $\G$. For an introduction to split reductive groups and maximal tori we refer the reader to \cite[Chapter I]{MILNE2} and \cite[Chapters 1 $\&$ 3]{CARTER}. See also Niemeyer--Praeger \cite[Section 3]{NP} for a description of the maximal tori in finite classical groups $\G^{\Fr_q}$  of Lie type.

\subsubsection{$\Fr_q$-stable maximal tori and characters of the Weyl group} \label{SubsectionToriPolys}

The $\G(\F_q)$-conjugacy classes in $\mathcal{T}^{\Fr_q}$ correspond to conjugacy classes in the Weyl group $B_n$. Lehrer observed  that this implies that, in principle, functions which are defined on $\Fr_q$-stable maximal tori (for instance, functions which count rational tori) may be described in terms of the character theory of $B_n$. 
This correspondence between $\G(\F_q)$-conjugacy classes  in $\mathcal{T}^{\Fr_q}$ and conjugacy classes in $B_n$ is defined as follows. Consider a torus $T\in \mathcal{T}^{\Fr_q}$, so $T=gT_0g^{-1}$ for some $g\in \G$ . Since $$gT_0g^{-1}=T=\Fr_q(T)=\Fr_q(g)\Fr_q(T_0)\Fr_q(g)^{-1}=\Fr_q(g)T_0\Fr_q(g)^{-1}$$ it follows that $\big(g^{-1}\cdot\Fr_q(g)\big)\in N(T_0)$. We denote by $w_T$ the element in the Weyl group $W(T_0) \cong B_n$ that is given by the projection of $\tilde{w}_T=g^{-1}\cdot\Fr_q(g)$  onto the quotient $W(T_0) = N(T_0)/T_0$. Since $\G$ is split, it turns out that each $\Fr_q$-stable maximal torus  determines  a conjugacy class in the Weyl group $B_n$ (\cite[Proposition 3.3.2]{CARTER}). 
Hence, given a class function $\chi$ on $B_n$, we can regard $\chi$ as a function on $\Fr_q$-stable maximal tori of $\G$, by defining
$$\chi(T):=\chi(w_T),\ \ \ \ \text{for }T\in\mathcal{T}^{\text{F}_q}.$$
This correspondence between conjugacy classes of tori and conjugacy classes in the Weyl group is illustrated in Example \ref{ToriExample} for  $\Sp_2(\overline{\F_p})$. 

\begin{example} {\bf ($\Fr_q$-stable maximal tori of $\Sp_2$).} \label{ToriExample}
For the algebraic group $\Sp_2(\overline{\F_p})=\mathrm{SL}_2(\overline{\F_p})$, the $\Fr_q$-stable maximal torus  $$T_0=\left\{ \left[ {\begin{array}{cc}\lambda_1& 0 \\ 0& \lambda_1^{-1} \end{array} } \right] \; \middle| \; \lambda_1\in\overline{\F_p}^\times\right\}$$
 is a split torus corresponding to the identity  coset $\left[ {\begin{array}{cc}1& 0 \\ 0& 1 \end{array} } \right]T_0$ in the Weyl group $W(T_0)$, that is, the identity element of the Weyl group $B_1$. On the other hand, given an element $\epsilon\in\F_q^\times$ which is not a square in $\F_q$, consider the abelian subgroup of $\SL_2(\overline{\F_p})$
$$T_\epsilon=\left\{ \left[ {\begin{array}{cc}x& y \\ \epsilon y& x \end{array} } \right]  \; \middle| \;  x,y\in\overline{\F_p},\; x^2-\epsilon y^2=1\right\}.$$
Choose a square root $\sqrt{\epsilon} \in \overline{\F_p} $ of $\epsilon$. If we take $g=\frac{1}{2\sqrt{\epsilon}} \left[ {\begin{array}{ccc}1& -1  \\ \sqrt{\epsilon}& \sqrt{\epsilon} \end{array} } \right] \in \SL_2(\overline{\F_p})$, then 
$$g^{-1}\cdot\left[ {\begin{array}{cc}x& y \\ \epsilon y& x \end{array} } \right]\cdot g= \left[ {\begin{array}{ccc}x+y\sqrt{\epsilon}& 0 \\ 0& x-y\sqrt{\epsilon}$$ \end{array} } \right],$$
therefore $g^{-1}T_\epsilon g= T_0$ and $T_\epsilon$ is a $\Fr_q$-stable maximal torus.
The matrix
$$ \tilde{w}_{T_{\epsilon}}=g^{-1}\cdot\Fr_q(g) = \frac{1}{2\sqrt{\epsilon}} \begin{bmatrix} \sqrt{\epsilon} &1 \\ -\sqrt{\epsilon} &1 \end{bmatrix} \begin{bmatrix}1 &-1 \\ -\sqrt{\epsilon} &-\sqrt{\epsilon} \end{bmatrix} = \begin{bmatrix}0 &-1 \\ -1 &0 \end{bmatrix} $$ 
projects to the coset $\left[ {\begin{array}{ccc}0& 1 \\ 1& 0 \end{array} } \right] T_0$ in the Weyl group $W(T_0)$. Hence $w_{T_\epsilon}$ is the negative 1-cycle $(1\ \overline{1})\in B_1$.
\end{example}

\subsubsection{The coinvariant algebra and statistics on maximal tori}  \label{SubsectionCoinvariantAlg}

There is a beautiful connection between the conjugacy-invariant functions $\chi(T)$ on the maximal tori of $\G$ defined over $\F_q$, and the topological structure of certain complex varieties related to $\G(\C)$. We now describe the complex cohomology algebra of these varieties, the generalized complete complex flag manifolds in type B and C. 

Let $V \cong \C^n$ denote the canonical complex representation of the hyperoctahedral group $B_n$ by signed permutation matrices. Then the symmetric powers Sym$(V)$ are isomorphic to $\C[x_1, \ldots, x_n]$, and we denote by $I_n$ the homogeneous ideal generated by the $B_n$--invariant polynomials with constant term zero.  The  \emph{complex type B/C coinvariant algebra} is defined as

$$R_n \cong \C[x_1, \ldots, x_n] / I_n $$
Let $R^{d}_n$ denote the $d^{th}$ graded piece of $R_n$.\\

Borel  \cite{Borel1953} proves that $R^*_n$ is isomorphic as a graded $\C[B_n]$-algebra to the cohomology of the {\it generalized complete complex flag manifolds} in type $B$  and $C$. The cohomology groups are supported in even cohomological degree; this isomorphism multiplies the grading by $2$. We recall the definition of these varieties: in type $B$, let $\B^B_n$ be a Borel subgroup of $\SO_{2n+1}(\C)$. Then the associated generalized flag manifold is $$\SO_{2n+1}(\C) / \B^B_n=\{0\subseteq V_1\subseteq\ldots\subseteq V_{2n+1}=\C^{2n+1}|\ \dim_\C V_m=m,\  q(V_i,V_{2n+1-i})=0\},$$
the variety of complete flags equal to their orthogonal complements. In type $C$,  take a Borel subgroup  $\B^C_n$ of $\Sp_{2n}(\C)$. The associated generalized flag manifold is $$ \Sp_{2n}(\C)/ \B^C_n=\{0\subseteq V_1\subseteq\ldots\subseteq V_{2n}=\C^{2n}|\ \dim_\C V_m=m,\ \omega(V_i,V_{2n-i})=0\},$$ the variety of complete flags equal to their symplectic complements.

A result of Lehrer \cite[Corollary  1.10']{LEHRER_TORI} relates the characters of $R^d_n$ to the space of maximal tori.  In the case of the split reductive groups   $\Sp_{2n}$ and $\SO_{2n+1}$ defined over $\overline{\F_p}$  with $2n^2$ roots, Lehrer's result specializes to the following formula. 

\begin{thm}[Corollary  1.10' \cite{LEHRER_TORI}]\label{LEHFOR}
	Let $q$ be  an integral power of a prime $p$ and let $\G$ be the linear algebraic group $\Sp_{2n}(\overline{\F_p})$ or $\SO_{2n+1}(\overline{\F_p})$. If $\chi$ is a class function on $B_n$, then

	
	\begin{equation}\label{COUNT2} 
	\sum_{T\in\mathcal{T}^{\Fr_q}} \chi(T)=q^{2n^2}\sum_{d\geq 0}  	q^{-d} \big\langle\chi, R_n^d\big\rangle_{B_n},
	\end{equation}
	 where  $\mathcal{T}^{\Fr_q}$ is the set of $\Fr_q$-stable maximal tori of $\G$  and $R^d_n$ is the $d^{th}$-graded piece  of the complex coinvariant algebra $R_n$ in type B/C.	
\end{thm}  

We remark that, in this formula, $R_n^d$ is generated in graded degree $d=1$, and not (as with the cohomology algebra) supported only in even degrees.

Formula (\ref{COUNT2}) in Theorem \ref{LEHFOR} describes a deep relationship between the maximal tori over finite fields and the cohomology of the topological spaces related to $\G(\C)$; with these results we will (in Sections \ref{SectionAsymptoticResults} and \ref{SectionConcreteStatistics}) relate representation-theoretic stability results on the coinvariant algebras to stability results for point-counts on our varieties over finite fields.

\subsubsection{Character polynomials and interpreting statistics} 

When $\chi$ is a character polynomial, the left-hand side of Formula (\ref{COUNT2}) is, in a sense, quantifying the numbers and decompositions of maximal tori over $\F_q$. To make this precise, we first consider the statistics for $\Fr_q$-stable maximal tori in $\Sp_{2n}(\overline{\F_q})$ and  $\SO_{2n+1}(\overline{\F_q})$ that correspond to the class functions $X_r$ and $Y_r$.

Let $T\in \mathcal{T}^{\Fr_q}$ as before with $T_0=g^{-1}Tg$ for $g\in\G$.
Consider the ordered set of eigenvectors $v_1,\ldots,v_n,\overline{v}_1,\ldots,\overline{v}_n$ of $T$  given by  the columns of $g$. Since $T$ is defined over $\F_q$,  the Frobenius morphism $\Fr_q$ takes eigenvectors to eigenvectors. Then $\Fr_q$ acts on the set of lines in $\overline{\F_q}^{2n}$ $$\mathbf{L}_T=\{L_1,\ldots,L_n,\overline{L_1},\ldots,\overline{L_n}\}  \qquad \text{where $L_i:=$ span$_{\overline{\F_q}}(v_i)$ and $\overline{L_i}:=$ span$_{\overline{\F_q}}(\overline{v}_i)$.}$$ Moreover, since $\Fr_q(g)=g\tilde{w}_T$, the Frobenius morphism acts on $\mathbf{L}_T$ by the element $w_T\in B_n$ that permutes the lines $L_i$ and swaps pairs $L_i/ \overline{L_i}$.

Let $T\in\mathcal{T}^{\text{F}_q}$.  For each positive or negative $r$-cycle of  $w_T$, we can consider its corresponding support  $\{L_{i_1},\ldots, L_{i_r},\overline{L}_{i_1},\ldots, \overline{L_{i_r}}\}$ in $\mathbf{L}_T$ and the subtorus of $T_0$
$$T_{0_r}:=\big\{\text{diag}(1,\ldots\lambda_{i_1},\ldots,\lambda_{i_r},\ldots,1,1,\ldots\lambda_{i_1}^{-1},\ldots,\lambda_{i_r}^{-1},\ldots,1): \lambda_{i_j}\in \overline{\F_q}^\times\big\}.$$

Then $T_r:=gT_{0_r}g^{-1}$ is an $\Fr_q$-stable $r$-dimensional subtorus of $T$ irreducible  over $\F_q$. A torus defined over a field $\k$ is {\it irreducible}  if it is not isomorphic over $\k$ to a product of tori. Furthermore:

\begin{itemize}
	\item If the orbit  $\{L_{i_1},\ldots, L_{i_r},\overline{L_{i_1}},\ldots, \overline{L_{i_r}}\}$ corresponds to a positive $r$-cycle of $w_T$, then $(\Fr_q)^r=\Fr_{q^r}$ fixes each $L_{i_j}$ and each $\overline{L_{i_j}}$ .  This means that there is a matrix $g_r$  with entries in $\F_{q^r}$ 
	such that $g_r^{-1}T_r g_r=T_{0_r}$ and the subtorus $T_r$ splits over $\F_{q^r}$. 
	
	\item If the orbit  $\{L_{i_1},\ldots, L_{i_r}, \overline{L_{i_1}},\ldots,  \overline{L_{i_r}}\}$ corresponds to a negative $r$-cycle of $w_T$, then $(\Fr_q)^r=\Fr_{q^r}$ swaps the lines $L_{i_j}$ and $\overline{L}_{i_j}$.  This implies that no matrix $g_r$  with entries in $\F_{q^r}$ 
	 diagonalizes the subtorus $T_r$ and hence  $T_r$ does not split over $\F_{q^r}$. 
	\end{itemize}

Therefore,  if $T\in\mathcal{T}^{\text{F}_q}$,  for $r\geq 1$ the polynomial characters $X_r(T)$ and $Y_r(T)$  count the following:\\
	
\noindent $X_r(T) $ is the number of $r$-dimensional  $\Fr_q$-stable subtori of $T$ irreducible over $\F_q$ that split over $\F_{q^r}$,\\

	\noindent $Y_r(T)$ is the number of $r$-dimensional $\Fr_q$-stable subtori of $T$ irreducible over $\F_q$ that do not split over $\F_{q^r}$. \\

\noindent {\bf Example  \ref{ToriExample} continued} {\bf (Statistics for maximal tori in $\Sp_2$).} Recall the maximual torus $T_{\epsilon} \in \Sp_2(\overline{\F_p})$ from Example  \ref{ToriExample}: 
$$T_\epsilon=\left\{ \left[ {\begin{array}{cc}x& y \\ \epsilon y& x \end{array} } \right]  \; \middle| \;  x,y\in\overline{\F_p},\; x^2-\epsilon y^2=1\right\},$$ 
The torus $T_\epsilon$ is diagonalized by the matrix $\frac{1}{2\sqrt{\epsilon}} \left[ {\begin{array}{ccc}1& -1  \\ \sqrt{\epsilon}& \sqrt{\epsilon} \end{array} } \right]$ in $\Sp_2(\overline{\F_p})$, and is invariant under the action of Frobenius. It corresponds to the nontrivial conjugacy class $w_{T_{\epsilon}} = (1 \, \oo1) \in B_1$. 

The Frobenius morphism $\Fr_q$ acts by transposing the eigenspaces
		$$L_1=\text{span}_{\overline{\F_p}}\left(\left[ {\begin{array}{c}1 \\ \sqrt{\epsilon} \end{array} } \right]\right)\text{ \  \  and  \  \  }\overline{L}_1=\text{span}_{\overline{\F_p}}\left(\left[ {\begin{array}{c}-1 \\ \sqrt{\epsilon} \end{array} } \right]\right).$$ 
	Correspondingly, no matrix that diagonalizes $T_\epsilon$ can have entries in $\F_q$. The one--dimensional torus $T_\epsilon$ is irreducible and does not split over $\F_{q}$. This is consistent with its statistics $$X_1(T_{\epsilon})=X_1((1\, \oo1))=0 \qquad \text{ and } \qquad Y_1(T_{\epsilon})=Y_1((1\, \oo1))=1.$$

\subsection{Asymptotic results} \label{SectionAsymptoticResults}

In this section we prove Theorem \ref{ASYM_TORI}, a stability result for asymptotic polynomial statistics on $\Fr_q$-stable maximal tori of the symplectic and special orthogonal groups.

\begin{thm}[\bf Stability of maximal tori statistics]\label{ASYM_TORI}
		Let $q$ be an integral power of a prime $p$. For $n\geq 1$, denote by $\mathcal{T}_n^{\Fr_q}$ the set of $\Fr_q$-stable maximal tori for either  $\Sp_{2n}(\overline{\F_p})$ or  $\SO_{2n+1}(\overline{\F_p})$. Let  $R^d_m$ denote the $d^{th}$-graded piece  of the complex coinvariant algebra $R^*_m$ in type B/C.
 If $P \in \C[X_1, Y_1, X_2, Y_2, \ldots ]$ is any hyperoctahedral character polynomial, then the normalized statistic  $q^{-2n^2} \sum_{T \in \mathcal{T}_n^{\Fr_q}} P(T)$ converges as $n\to\infty$. In fact,  $$ \lim_{n \to \infty} q^{-2n^2} \sum_{T \in \mathcal{T}_n^{\Fr_q}} P(T) \; = \; \sum_{d=0}^{\infty} \frac{\lim_{m \to \infty} \langle P_m, R^d_m \rangle_{B_m} }{q^d},$$
and the series in the right hand converges. 
	\end{thm}

To prove this theorem we first establish a convergence result for characters of coinvariant algebras of type B/C.

Let $C^*_n$ denote the complex polynomial algebra $\C[x_1, \ldots, x_n]$ on $n$ variables, with generators $x_i$ in graded degree $d=1$.  These polynomial rings form a complex $\FI_{BC}$--algebra under the natural inclusions $C^*_n \hookrightarrow C^*_{n+1} $, which is generated as an $\FI_{BC}$--algebra in $\FI_{BC}$--degree $n=1$ by $C^1_1 = \langle x_1 \rangle$. Our sequence of algebras is then generated by the $\FI_{BC}$--module of canonical signed permutation representations, 
$$C^1_{\bullet} \cong M_{BC}(\varnothing, \Y{1}\; )_{\bullet} := M_{BC}({\bf 1})_{\bullet} \otimes_{\C[B_1]} V_{(\varnothing,\; \Y{1}\;)} \cong \langle x_1, \ldots, x_{\bullet} \rangle.$$
The coinvariant algebra $R^*_n \cong C^*_n / I_n $ does not admit an $\FI_{BC}$--module structure; the $\FI_{BC}$--module structure on the polynomial rings $C^*_n$  does not respect the ideals $I_n$. Instead, the coinvariant algebras are $\C$--algebras over $\FI_{BC}^{\text{op}}$, induced by the maps  
\begin{align*}
C^*_{n+1} &\longrightarrow  C^*_n\\ 
 x_i &\longmapsto  \left\{ \begin{array}{ll} x_i, & i=1, \ldots, n
 \\ 0, & i = n+1.  \end{array} \right. 
\end{align*}
The reader may verify that the duals  $\widehat{R^d_{\bullet}}$ of each graded piece do form finitely generated sub--$\FI_{BC}$--modules of the dual modules $\widehat{C^d_{\bullet}} \cong C^d_{\bullet}$; see work of the second author \cite[Definition 5.14, Proposition 5.15, and Section 6]{FIW1} for details. For the following result, however, we do not need this $\FI_{BC}$--module structure; we only need the observation that for each $n$ and $d$, the $B_n$--representation $R^d_n$ is a subquotient of the homogeneous polynomial space $C^d_n$. 

\begin{prop} \label{CONVCOIN} Let $R^d_n$ the $d^{th}$--graded piece  of the complex type B/C coinvariant algebra $R_n$.   Then for any hyperoctahedral character polynomial $P \in \C[X_1, Y_1, X_2, Y_2, \ldots ]$ the following sum converges absolutely. $$ \sum_{d= 0}^\infty   \frac{ \lim_{n \to \infty}  \langle P_n, R^d_n \rangle_{B_n}}{q^d}.$$
\end{prop}
\begin{proof} We have observed that $R^*_n$ is an $\FI^\text{op}_{BC}$--algebra quotient of the free graded-commutative algebra $C^*$ on $$M_{BC}(\varnothing, \Y{1})_{\bullet} := M_{BC}({\bf 1})_{\bullet} \otimes_{\C[B_1]} V_{(\varnothing, \; \Y{1} \;)} $$
The $\FI_{BC}$--module $M_{BC}(\varnothing, \Y{1}\;)_{\bullet}$ is, in the notation of Definition \ref{DefnM(m)}, the submodule of $M_{BC}({\bf 1})$ on the generators $x_i := e_i - e_{\overline{i}}$. By Remark \ref{RemCoFI}, the result follows from Theorem \ref{LIMIT}.
\end{proof}

\begin{rem}\label{COTYPEA}This result was proved by Church--Ellenberg--Farb in type A \cite[Theorem 5.6]{CEFPointCounting} using in part work of Chevalley \cite{ChevalleyInvariants} on the structure of the type A coinvariant algebra. Our Theorem \ref{LIMIT} allows for a combinatorial proof of \cite[Theorem 5.6]{CEFPointCounting} using only the fact that the coinvariant algebras are quotient of the dual of the free commutative $\FI$--algebra $\C[x_1, x_2, \ldots, x_n]$ generated by $M({\bf 1})_n = \langle x_1, x_2, \ldots, x_n \rangle$. This proof demonstrates that this convergence theorem does not depend on any deeper structural features of the coinvariant algebra. \end{rem}

	\begin{rem}\label{GEN1} Work of the second author \cite[Theorem 6.1 and Corollary 6.3]{FIW1} observes that, for each graded--degree $d$, the $B_n$--representations $\widehat{R^d_n}$ dual to $R^d_{n}$  assemble to form a finitely generated $\FI_{BC}$--module $\widehat{R^d_{\bullet}}$, a submodule of the sequence of homogenous polynomials $\widehat{C^d_{\bullet}}$. The $\FI_{BC}$--module $\widehat{C^d_{\bullet}}$, and hence $\widehat{R^d_{\bullet}}$, has weight at most $d$. Hence for each $d$ there is some $D_d \in \Z$ such that the characters of $\widehat{R^d_{\bullet}}$ are given by a character polynomial of degree $\leq d$  for all $n \geq D_d$ \cite[Theorem 4.16]{FIW2}. Since $B_n$--representations are self-dual, there is an isomorphism of representations $\widehat{R^d_n} \cong R^d_n$, and this character polynomial gives the characters of $R^d_n$ for all $n \geq D_d$.  It follows from Proposition \ref{STABPOL} that  
	$$  \lim_{m \to \infty} \langle P_m, R^d_m \rangle_{B_m} = \langle P_n, R^d_n \rangle_{B_n}, \qquad\text{ for any } n\geq \max\{D_d, d+\deg(P) \}.$$
\end{rem}

\begin{proof}[Proof of Theorem \ref{ASYM_TORI}]
We adapt the arguments given by Church--Ellenberg--Farb \cite[Theorem 3.13]{CEFPointCounting}. 
From  Lemma \ref{SUBEXP}, there exist a  function $F_P(d)$ independent of $n$ and subexponential in $d$ such that $|\langle P_n, R^d_n \rangle_{B_n}|\leq F_P(d)$ for all $n$. Then $$\left\vert \lim_{m\to\infty}\langle P_m, R^d_m \rangle_{B_m}\right\vert \leq F_P(d).$$ 
Let $\epsilon>0$. The series $\displaystyle \sum_{d\geq 0} \frac{F_P(d)}{q^d}$ converges absolutely, so  there exist some $I\in\mathbb{N}$ such that  $$\sum_{d\geq I+1} \frac{F_P(d)}{q^d}<\epsilon/2.$$ 
Using the notation from Remark \ref{GEN1}, we choose $N = \max\{D_1,D_2,\ldots, D_{I}, I+\deg(P) \}$. Then 

$$  \lim_{m \to \infty} \langle P_m, R^d_m \rangle_{B_m} = \langle P_n, R^d_n \rangle_{B_n}\qquad \text {for $d\leq I$ and n$\geq N$}.$$
From Proposition \ref{CONVCOIN}, the series $$\sum_{d=0}^{\infty} \frac{\lim_{m \to \infty} \langle P_m, R^d_m \rangle_{B_m} }{q^d}$$ converges to a limit $L<\infty$. On the other hand, by Theorem \ref{LEHFOR}, we have that 
$$ q^{-2n^2} \sum_{T \in \mathcal{T}_n^{\Fr_q}} P(T) \; = \; \sum_{d\geq 0} \frac{ \langle P_n, R^d_n \rangle_{B_n} }{q^d}.$$ 

Therefore, if $n\geq N$

\begin{align*}
	\left|L-  q^{-2n^2} \sum_{T \in \mathcal{T}_n^{\Fr_q}} P(T)\right|  &  =  \left|\sum_{d\geq I+1} \frac{ \lim_{m \to \infty} \langle P_m, R^d_m \rangle_{B_m} -\langle P_n, R^d_n \rangle_{B_n} }{q^d}\right|\\
	& \leq  \sum_{d\geq I+1} \frac{ F_P(d)+F_P(d) }{q^d} \\ & < \epsilon
\end{align*} 

\end{proof}

\subsection{Some statistics for $\Fr_q$-stable maximal tori in $\Sp_{2n}(\overline{\F_q})$ and  $\SO_{2n+1}(\overline{\F_q})$} \label{SectionConcreteStatistics}

In this section we compute some examples of statistics of the form given in Theorem \ref{ASYM_TORI}.
We use a result on the decomposition of the coinvariant algebra $R^*_n$ as a $B_n$--representation; see Stembridge \cite[Formula (2.2), Theorem 5.3]{StembridgeReflectionGroups}, analogous to the approach taken by Church--Ellenberg--Farb \cite[Theorems 5.8, 5.9, and 5.10]{CEFPointCounting}. These computations could also be accomplished with generating functions, using the approach of Fulman  \cite{FULMAN} in type A. \medskip

\noindent{\bf Decomposing the $B_n$--representation $R^d_n$.} The decomposition of the graded pieces $R^d_n$ into irreducible $B_n$--representations is described by Stembridge \cite{StembridgeReflectionGroups} in terms of their relationships to data called the \emph{fake degrees} of $B_n$. Given a double partition $\y=(\y^+, \y^-)$ of $n$, the multiplicity of irreducible representation $V_{\y}$ in $R^d_n$ is computed by the following formula. A \emph{double standard tableau} of shape $\y$ is a bijective labelling of the Young diagrams $\y^+$ and $\y^-$ with the digits $1$ through $n$, such that the numbers are strictly increasing in each row and column. Define the \emph{flag descent set} $D(T)$ of a double standard tableau $T$ as follows:  draw the tableaux of shape $\y^+$ and $\y^-$ in the plane, with the tableau of shape $\y^-$ placed above and to the right of tableau of shape $\y^+$. Then $$D(T) = \left\{ \quad j \quad \middle| \quad  (j+1) \text{ appears in a row below $j$} \quad  \right\}. $$
and the \emph{flag major index} of the double standard tableau $T$ is the quantity $$ f(T) = 2 \left( \sum_{ j \in D(T)} j \right) + |\y^-|.$$
Then the multiplicity of the irreducible representation $V_{\y}$ in the $d^{th}$-graded piece of the coinvariant algebra $R^d_n$ is given by the number of standard double tableaux of shape $\y$ and flag major index $d$. 

\begin{thm}[See Stembridge \cite{StembridgeReflectionGroups} Formula (2.2), Theorem 5.3] \label{ThmStembridge}
	$$	\langle R_n^d, V_{\y} \rangle_{B_n} = \#\left\{\quad T \quad \middle| \quad   \text{  $T$ standard double tableau of shape $\y$ with $f(T)=d$ } \quad  \right\}  .$$	
\end{thm}

We can use Theorem  \ref{ThmStembridge} to compute some examples of the normalized statistics for maximal tori whose convergence is guaranteed by Theorem \ref{ASYM_TORI}.  A formula identifying stable sequences of irreducible $B_n$--representations $V_{\y}$ with a character polynomial $P$ is given in Wilson \cite[Theorem 4.11]{FIW2}. Some stable multiplicities that we compute  below, using Theorem  \ref{ThmStembridge}, are summarized in Table \ref{TableStableMultiplicities}. Dots represent zero. 

\begin{table} \caption{Some stable multiplicities of irreducible $B_n$--representations in the coinvariant alegebra}
\label{TableStableMultiplicities}
\begin{adjustbox}{max width=1.05\textwidth}

\begin{tabular}{|c || c || c | c | c | c | c | c | c | c | c | c | c | c |c| c| c| c| c|  } 
\hline &&&&&&&&&&&&&&&&& \\	
	Irreducible $B_n$ & Hyperocthahedral & $d=0$  &  1  & 2 & 3 & 4 & 5 & 6 & 7 & 8 & 9 & 10 & 11 & 12 & 13 & 14 & 15  \\
		representation & character polynomial &  &  &&&&&&&&&&&&&& \\
		&&&&&&&&&&&&&&&&&  \\
	\hline &&&&&&&&&&&&&&&&&  \\
	$ V((n), \varnothing)$  & 1 &  1  & $\cdot$ & $\cdot$ & $\cdot$ & $\cdot$ & $\cdot$ & $\cdot$ & $\cdot$ & $\cdot$ & $\cdot$ & $\cdot$  & $\cdot$ & $\cdot$ & $\cdot$ & $\cdot$ & $\cdot$  \\
	&&&&&&&&&&&&&&&&& \\	
	$V( (n-1), (1))$ & $ X_1 - Y_1 $& $\cdot$ &   1  & $\cdot$ & 1 & $\cdot$ & 1 & $\cdot$ & 1   & $\cdot$ & 1 & $\cdot$ & 1 & $\cdot$ & 1 & $\cdot$ & 1 \\
	&&&&&&&&&&&&&&&&& \\	
	$V( (n-1, 1), \varnothing)$ & $X_1 + Y_1-1 $ & $\cdot$ & $\cdot$ &   1  & $\cdot$ & 1 & $\cdot$ & 1 & $\cdot$ & 1   & $\cdot$ & 1 & $\cdot$ & 1 & $\cdot$ & 1 & $\cdot$  \\
	&&&&&&&&&&&&&&&&& \\	
	
	$V((n-2), (2))$  & $ \displaystyle { X_1 \choose 2} + { Y_1 \choose 2}  - X_1Y_1 +X_2-Y_2 $ &    $\cdot$ & $\cdot$ & 1 & $\cdot$ & 1  & $\cdot$ & 2 & $\cdot$ & 2 & $\cdot$ & 3 & $\cdot$ & 3 & $\cdot$ & 4 & $\cdot$   \\	
	&&&&&&&&&&&&&&&&& \\	
	
	$V((n-2), (1, 1))$  & $ \displaystyle { X_1 \choose 2} + { Y_1 \choose 2}  - X_1Y_1 -X_2+Y_2 $ &   $\cdot$  &  $\cdot$  & $\cdot$ & $\cdot$ & 1 & $\cdot$ & 1  & $\cdot$ & 2 & $\cdot$ & 2 & $\cdot$ & 3 & $\cdot$ & 3 & $\cdot$  \\
	&&&&&&&&&&&&&&&&& \\	
	$V((n-2, 1, 1), \varnothing)$  & $ \displaystyle  { X_1 + Y_1 \choose 2} - X_2 - Y_2 - X_1 - Y_1 +1$ &  $\cdot$ &  $\cdot$ &  $\cdot$  & $\cdot$  & $\cdot$ & $\cdot$ & 1 & $\cdot$ & 1  & $\cdot$ & 2 & $\cdot$ & 2 & $\cdot$ &3&$\cdot$  \\
	&&&&&&&&&&&&&&&&& \\	
	
	$V( (n-2,1), (1) )$ & $ X_1^2-Y_1^2-2X_1+2Y_1 $ & $\cdot$  & $\cdot$ & $\cdot$ & 1 & $\cdot$ & 2 & $\cdot$ & 3 & $\cdot$ & 4 & $\cdot$ & 5 & $\cdot$ & 6 & $\cdot$ & 7  \\
	
	&&&&&&&&&&&&&&&&&  \\  \hline
\end{tabular} 
\end{adjustbox}
\end{table}

\medskip

Our first computation is a classical result due to Steinberg (see \cite[Corollary 14.16]{STEIN}).

\begin{thm}[\bf Number of $\Fr_q$-stable maximal tori]  \label{ThmNumMaxTori} The number of $\Fr_q$-stable maximal tori of $\Sp_{2n}( \overline{\F_q})$ and  of $\SO_{2n+1}(\overline{\F_q})$ is $q^{2n^2}$.
\end{thm}
\begin{proof}
We can compute the number of $\Fr_q$-stable maximal tori using Theorem \ref{LEHFOR} by taking $\chi$ to be the trivial class function on $B_n$. The sequence of trivial representations are given by the character polynomial $\chi=1$,  and are encoded by the double partition $\y= \left( \overbrace{ \Y{3} \cdots \Y{1} }^{n} \; , \;	\varnothing
\; \right).$ The single double standard tableau  of shape $\y$ is $$ T= \left( 
\begin{ytableau}
1 & 2 & 3 &  \none & \none[\cdots] &  \none
& n \end{ytableau}
\; , \; \varnothing
\right)$$
which has descent set $D(T)= \varnothing$ and flag major index $f(T) =0$. Thus by Theorem \ref{ThmStembridge} the trivial representation appears with multiplicity one in degree $0$ and does not occur in positive degree. 
\end{proof} 

\begin{rem}
Observe that the normalized formula
$$\frac{ \sum_{T \in \mathcal{T}_n^{\Fr_q}} P(T)}{ q^{2n^2}}=\frac{ \sum_{T \in \mathcal{T}_n^{\Fr_q}} P(T)}{ |\mathcal{T}_n^{\Fr_q}|}$$
corresponds to the average of the statistic $P(T)$ over all maximal tori $T$.  Using Theorem \ref{ASYM_TORI} we obtain asymptotics of these expected values.
\end{rem}

\begin{prop}[\bf Expected number of $1$-dimensional $\Fr_q$-stable subtori] \label{Av1DTori} The expected number of $1$-dimensional $\Fr_q$-stable subtori of a random $\Fr_q$-stable maximal torus in $\Sp_{2n}(\overline{\F_q})$ or in $\SO_{2n+1}(\overline{\F_q})$ equals $$1+\frac{1}{q^2}+\frac{1}{q^4}+\cdots + \frac{1}{q^{2n-2}} = \frac{\left( q^2 - \frac{1}{q^{2(n-1)}}\right)}{(q^2-1)} \quad \xrightarrow{n \to \infty} \quad \frac{q^2}{q^2-1}.$$
\end{prop}
\begin{proof}
To count the total number of $1$-dimensional $\Fr_q$-stable subtori for all maximal tori in $\mathcal{T}^{\Fr_q}$, we apply Formula  (\ref{COUNT2}) to the character polynomial $P=X_1+Y_1$. The goal is to show  that $\big\langle P, R_n^d\big\rangle_{B_n}= 1$ for $d=0, 2, 4,\ldots,2n-2$ and $0$ otherwise.

The pullback of the  $(n-1)$--dimensional standard $S_n$--representation to $B_n$ has character polynomial  $\chi = X_1 + Y_1-1$ and corresponds to the double partition $$\y= \left( \overbrace{ \Y{3,1} \substack{\;\; \cdots \;\; \Y{1} \\ \;} }^{n-1} \; , \;	\varnothing \; \right) .$$ There are $(n-1)$ possible double standard tableaux of shape $\y$, each determined by the letter $i$ in the second row. Let $$T_i= \left( \quad
\begin{ytableau}
1 & 2 &  \none[\cdots] & {\scriptstyle i-1 } & {\scriptstyle i+1 } & \none[\cdots] 
& n \\ i  \end{ytableau}
\quad , \quad \varnothing \quad 
\right) \qquad \text{for $i=2, \ldots, n$.} $$
\noindent Note that we require $i>1$ for $T_i$ to form a valid standard tableau. Then  $D(T_i)=\{i-1\}$ for $i=2, \ldots, n$.  The flag major indices are $f(T_i)=2i-2$ for $i=2, \ldots, n$, and there is one copy of $V_{(n-1, 1), \varnothing}$ in each even degree $2, 4, \ldots, 2n-2$.  

The character polynomial $(X_1+Y_1)$  is the sum $(X_1+Y_1) = (X_1+Y_1-1) + 1$, so combining this result for multiplicities of the representation $V_{( (n-1, 1), \varnothing )}$ and the result for $V_{((n), \varnothing )}$ of Theorem \ref{ThmNumMaxTori} gives the desired formula. 

\end{proof}

\begin{prop}[\bf  Expected number of split $1$-dimensional $\Fr_q$-stable subtori]
The expected number of split $1$-dimensional $\Fr_q$-stable subtori of of a random $\Fr_q$-stable maximal torus in $\Sp_{2n}(\overline{\F_q})$ or in $\SO_{2n+1}(\overline{\F_q})$ equals $$\frac{1}{2}\left(1+\frac{1}{q}+\frac{1}{q^2}+\frac{1}{q^3}+\cdots  +\frac{1}{q^{2n-1}}\right)  = \frac{\left(q - \frac{1}{q^{2n-1}} \right)}{2(q-1)}\quad \xrightarrow{n \to \infty} \quad \frac{q}{2(q-1)}.$$ 
\end{prop}
\begin{proof}
We wish to evaluate the character polynomial $P=X_1=\frac{1}{2}[(X_1+Y_1)+(X_1-Y_1)]$ in formula (\ref{COUNT2}).
The canonical $n$--dimensional representation of $B_n$ by signed permutation matrices has character polynomial  $\chi = X_1 - Y_1$ and corresponds to the double partition $$\y= \left( \overbrace{ \Y{3} \cdots \Y{1} }^{n-1} \; , \;	\Y{1} \; \right) .$$ There are $n$ possible double standard tableaux of shape $\y$, each determined by the letter $i$ in $\y^-$. Let $$T_i= \left( \quad
\begin{ytableau}
1 & 2 &  \none[\cdots] & {\scriptstyle i-1 } & {\scriptstyle i+1 } & \none[\cdots]
& n \end{ytableau}
\quad , \quad	\begin{ytableau}
i  \end{ytableau} \quad 
\right) \qquad \text{for $i=1, \ldots n$.} $$
Then $D(T_n)=\varnothing$, and  $D(T_i)=\{i\}$ for $i=1, \ldots, n-1$.  The flag major indices are

$$ f(T_i) = \left\{ \begin{array}{ll} 1, & i=n \\ 2i+1, & i=1, 2, \ldots, (n-1) \end{array} \right.$$
and then there is a single copy of $V_{((n-1), (1))}$ in each odd degree $1, 3, \ldots, 2n-1$. 

By combining this result with our computation for $X_1+Y_1$ in Proposition \ref{Av1DTori} above, we conclude that  $\big\langle 2X_1,R_n^d\big\rangle_{B_n}= 1$ for $d=0, 1,2,3, \ldots,2n-1$ and  $\big\langle 2X_1,R_n^d\big\rangle_{B_n}= 0$ otherwise. 
\end{proof}

\begin{cor}[\bf Expected number of eigenvectors in $\F_q^{2n}$]
The expected number of simultaneous eigenvectors in $\F_q^{2n}$ of a random $\Fr_q$-stable maximal torus in $\Sp_{2n}(\overline{\F_q})$ equals $$1+\frac{1}{q}+\frac{1}{q^2}+\frac{1}{q^3}+ \cdots  +\frac{1}{q^{2n-1}}$$
\end{cor}
\begin{proof}
The fixed points of $w_T$ correspond to eigenvectors $v_i$ of $T$ in $\F_q^{2n}$. Since the eigenvectors come in inverse pairs ($v_i/\overline{v}_i$) for $\Fr_q$-stable maximal tori in $\Sp(\overline{\F_q})$,  we apply Formula (\ref{COUNT2}) to the character polynomial $P=2X_1$. The result follows from the above computation.
\end{proof}


\begin{prop}[\bf Reducible v.s. Irreducible $\Fr_q$-stable $2$-dimensional subtori] Given a $\Fr_q$-stable torus $T$ in $\Sp_{2n}(\overline{\F_q})$ or in $\SO_{2n+1}(\overline{\F_q})$, let $\mathcal{R}_n(T)$ denote the number of reducible $2$-dimensional $\Fr_q$-stable subtori of $T$ and  $\mathcal{I}_n(T)$ denote the number of irreducible $2$-dimensional $\Fr_q$-stable subtori of $T$. Then the expected value of the function $\mathcal{R}_n-\mathcal{I}_n$ over all $\Fr_q$-stable maximal tori of $\Sp_{2n}(\overline{\F_q})$ or $\SO_{2n+1}(\overline{\F_q})$ is given by
$$ \frac{\left( q^4 - \frac{1}{q^{2n}} \right) \left( 1 - \frac{1}{q^{2(n-1)}} \right) }{\left(q^2-1 \right)\left(q^4 -1\right)} $$
and converges to the sum
 $$\frac{1}{q^{2}}+\frac{1}{q^{4}}+\frac{2}{q^{6}}+\frac{2}{q^{8}}+\frac{3}{q^{10}}+\frac{3}{q^{12}}+\ldots+\frac{ \lfloor\frac{2d-2}{4}\rfloor}{q^{2d-4}} +\ldots \quad = \quad \frac{q^4}{(q^2-1)(q^4-1)}$$
as $n$ tends to infinity.
 \end{prop}

\begin{proof}
To count the difference in the number of reducible and irreducible $2$-dimensional $\Fr_q$-stable subtori of $T\in\mathcal{T}_n^{\Fr_q}$, we compute $$P(T)={ X_1(T) + Y_1(T) \choose 2} - \Big(X_2(T) + Y_2(T)\Big)$$  and employ  Formula (\ref{COUNT2}).  The irreducible $B_n$--representation $\bigwedge^2 V_{(n-1,1), \varnothing}$ has character $$\chi = { X_1 + Y_1 \choose 2} - (X_2 + Y_2) - (X_1 + Y_1) + 1$$ and associated double partition $$\y= \left( \overbrace{ \Y{3,1,1} \substack{\;\; \cdots \;\; \Y{1} \\ \; \\ \;}  }^{n-2} \; , \; \varnothing \; \right) .$$
Define 
$$T_{i,j}= \left( \quad \begin{ytableau}
1 & 2 &  \none[\cdots] & {\scriptstyle i-1 } & {\scriptstyle i+1 } & \none[\cdots] & {\scriptstyle j-1 } & {\scriptstyle j+1 } & \none[\cdots]
& n \\ i \\ j  \end{ytableau}
\quad , \varnothing \quad 
\right) \qquad \text{for $i < j, \quad i,j \in \{2, \ldots, n \}$.} $$ We have $D(T_{i,j}) = \{ i-1, j-1 \}$ for all $2 \leq i < j \leq n$, and so $f(T_{i,j}) = 2(i+j)-4$ for all  $2 \leq i < j \leq n$. Stably, it follows that the multiplicity is zero in degree $d$ if $d$ is odd, and for $d \geq 6$ even,  $f(T_{i,j})=d$ for the $\lfloor\frac{d-2}{4}\rfloor$ pairs  $(i,j) = (2, \frac{d}{2}), (3, \frac{d}{2}-1), \ldots, (\lfloor\frac{d+2}{4}\rfloor,  \frac{d}{2}-\lfloor\frac{d+2}{4}\rfloor +2)$. For a given $n \geq 2$, the expected value $ \displaystyle q^{-2n^2} \sum_{T \in \mathcal{T}_n^{\Fr_q}} \chi(T) $ of $\chi(T)$ is given in terms of the \emph{Gaussian binomial coefficient} $\displaystyle { n \choose 2}_{\frac1{q^2}} $ by the formula
\begin{align*} \; \sum_{d\geq 0} \frac{ \langle \chi, R^d_n \rangle_{B_n} }{q^d}= \frac1{q^6} { n \choose 2}_{\frac1{q^2}} := &  \frac{1}{q^6} \left( \frac{ \left(1-\frac1{q^{2n}}\right)\left(1-\frac{1}{q^{2(n-1)}}\right)}{\left(1-\frac1{q^2}\right)\left(1-\frac1{q^4}\right)}  \right) =   \frac{\left(1-\frac1{q^{2n}}\right)\left(1-\frac{1}{q^{2(n-1)}}\right)}{\left(q^2-1 \right)\left(q^4 -1\right)}   \end{align*}

As $n$ tends to infinity, the expected value converges to 
$$\frac{1}{q^6}+\frac{1}{q^{8}}+\frac{2}{q^{10}}+\frac{2}{q^{12}}+\frac{3}{q^{14}}+\frac{3}{q^{16}}+\ldots+\frac{ \lfloor\frac{2d-2}{4}\rfloor}{q^{2d}} +\ldots \quad = \quad \frac{1}{(q^2-1)(q^4-1)}$$

To obtain the result for the character polynomial
$${ X_1 + Y_1 \choose 2} - (X_2 + Y_2)  =\bigg[ { X_1 + Y_1 \choose 2} - (X_2 + Y_2) - (X_1 + Y_1) + 1 \bigg] + \bigg[ (X_1 + Y_1) - 1 \bigg]  $$
we combine this result with the result for irreducible representation $V_{(n-1,1), \varnothing}$ from Proposition \ref{Av1DTori}. We find that the desired formula is
\begin{align*} &\frac{\left(1-\frac1{q^{2n}}\right)\left(1-\frac{1}{q^{2(n-1)}}\right)}{\left(q^2-1 \right)\left(q^4 -1\right)}   + \frac{\left( q^2 - \frac{1}{q^{2(n-1)}}\right)}{(q^2-1)} -1 \\
&= \frac{\left( q^4 - \frac{1}{q^{2n}} \right) \left( 1 - \frac{1}{q^{2(n-1)}} \right) }{\left(q^2-1 \right)\left(q^4 -1\right)} \\
&  \xrightarrow{n \to \infty} \frac{q^4 }{\left(q^2-1 \right)\left(q^4 -1\right)} = \frac{1}{q^{2}}+\frac{1}{q^{4}}+\frac{2}{q^{6}}+\frac{2}{q^{8}}+\frac{3}{q^{10}}+\frac{3}{q^{12}}+\ldots
\end{align*}

\end{proof}

\begin{prop}[\bf Split  v.s. non-split  $\Fr_q$-stable irreducible $2$-dimensional subtori] The expected value of split minus non-split  $\Fr_q$-stable irreducible $2$-dimensional subtori over all $\Fr_q$-stable maximal tori of $\Sp_{2n}(\overline{\F_q})$ or $\SO_{2n+1}(\overline{\F_q})$ is given by
 $$   \frac{ q^2 \left(1-\frac1{q^{2n}}\right)\left(1-\frac{1}{q^{2(n-1)}}\right)}{2\left(q^4-1\right)} 
 \qquad \xrightarrow{n \to \infty} \qquad \frac{ q^2}{2\left(q^4-1\right)}.  $$
\end{prop} 

\begin{proof} We need to consider the character polynomial $X_2 - Y_2$ in Formula (\ref{COUNT2}). 

The irreducible $B_n$--representation $\Lambda^2 V_{(n-1), (1)}$ has character polynomial  $\displaystyle  \chi = { X_1 \choose 2} + { Y_1 \choose 2}  - X_1Y_1 -X_2+Y_2$ and corresponds to the double partition $$\y= \left( \overbrace{ \Y{3} \cdots \Y{1} }^{n-2} \; , \; \Y{1,1} \; \right) .$$ There are $\displaystyle { n \choose 2}$ possible double standard tableaux of shape $\y$. For $i < j$, with $i,j \in \{1, \ldots, n \}$, let $$T_{i,j}= \left( \quad
\begin{ytableau}
1 & 2 &  \none[\cdots] & {\scriptstyle i-1 } & {\scriptstyle i+1 } & \none[\cdots] & {\scriptstyle j-1 } & {\scriptstyle j+1 } & \none[\cdots]
& n  \end{ytableau}
\quad , \quad \begin{ytableau} i \\ j \end{ytableau} \quad 
\right)$$
Then 
\begin{align*}
&D(T_{i, n})=\{ i \} && f(T_{i, n}) = 2i+2 && \text{for} \quad i = 1, \ldots, n-1 \\
&D(T_{i, j}) = \{i, j\} && f(T_{i, j})=2i+2j+2 &&  \text{for} \quad  1 \leq i < j \leq n-1 
\end{align*}

Hence for even $d$ with $4 \leq d \leq 2n$, $f(T_{ \frac{d-2}{2}, n} ) = d$, and for even $d$ with $4 \leq d \leq 2n-2$, $f(T_{i,j}) = d$ for all $i<j<n$ such that  $i+j=\frac{d}{2}-1$. Stably there are $ \displaystyle \left\lfloor\frac{d}{4}\right\rfloor$ copies of $V_{(n-2), (1,1)}$ in each even degree $d \geq 4$ and zero copies in odd degree. For given $n \geq 2$, the expected value of the character polynomial statistic is
\begin{align*} \; \sum_{d\geq0} \frac{ \langle \chi, R^d_n \rangle_{B_n} }{q^d}= \frac1{q^4} { n \choose 2}_{\frac1{q^2}} := & \frac{ \frac1{q^4} \left(1-\frac1{q^{2n}}\right)\left(1-\frac{1}{q^{2(n-1)}}\right)}{\left(1-\frac1{q^2}\right)\left(1-\frac1{q^4}\right)} \\ 
& \xrightarrow{ n \to \infty} \frac{q^2}{\left(q^2-1\right)\left(q^4-1\right)} 
\end{align*}
The irreducible $B_n$--representation with character polynomial  $ \displaystyle \chi = { X_1 \choose 2} + { Y_1 \choose 2}  - X_1Y_1+X_2-Y_2$ and corresponds to the double partition $$\y= \left( \overbrace{ \Y{3} \cdots \Y{1} }^{n-2} \; , \; \Y{2} \; \right) .$$ There are $\displaystyle { n \choose 2}$ possible double standard tableaux of shape $\y$. For $i < j$, with $ i,j \in \{1, \ldots, n \}$, let $$T_{i,j}= \left( \quad
\begin{ytableau}
1 & 2 &  \none[\cdots] & {\scriptstyle i-1 } & {\scriptstyle i+1 } & \none[\cdots] & {\scriptstyle j-1 } & {\scriptstyle j+1 } & \none[\cdots]
& n  \end{ytableau}
\quad , \quad \begin{ytableau} i & j \end{ytableau} \quad 
\right) $$
Then 
\begin{align*}
&D(T_{n-1, n})=\varnothing && f(T_{n-1, n})=2 \\
&D(T_{i, n})=\{ i \} && f(T_{i, n}) = 2i+2 && \text{for} \quad i = 1, \ldots, n-2 \\
&D(T_{i-1, i}) = \{i\} && f(T_{i-1, i})=2i+2 &&  \text{for} \quad  i=2, \ldots, n-1 \\
&D(T_{i, j+1}) = \{i, j+1\} && f(T_{i, j+1})=2i+2j+4 &&  \text{for} \quad  1 \leq i < j \leq n-2 
\end{align*}
Stably, there this irreducible representation has multiplicity zero in odd degrees and multiplicity $ \displaystyle \left\lceil \frac{d}{4} \right\rceil$ in each even degree $d \geq 2$. For $n \geq 2$, the expected value $ \displaystyle q^{-2n^2} \sum_{T \in \mathcal{T}_n^{\Fr_q}} \chi(T) $ is
\begin{align*} \; \sum_{d\geq0}\ \frac{ \langle \chi, R^d_n \rangle_{B_n} }{q^d}= \frac1{q^2} { n \choose 2}_{\frac1{q^2}} := & \frac{ \frac1{q^2} \left(1-\frac1{q^{2n}}\right)\left(1-\frac{1}{q^{2(n-1)}}\right)}{\left(1-\frac1{q^2}\right)\left(1-\frac1{q^4}\right)}  =
\frac{ q^4 \left(1-\frac1{q^{2n}}\right)\left(1-\frac{1}{q^{2(n-1)}}\right)}{\left(q^2-1\right)\left(q^4 -1\right)} \\ 
& \xrightarrow{ n \to \infty} \frac{q^4}{\left(q^2-1\right)\left(q^4-1\right)} 
\end{align*}


Then for the character polynomial $$ P=(X_2 - Y_2) =  \frac12 \left[\left(  { X_1 \choose 2} + { Y_1 \choose 2}  - X_1Y_1 +X_2-Y_2 \right) -   \left( { X_1 \choose 2} + { Y_1 \choose 2}  - X_1Y_1 -X_2+Y_2  \right) \right]$$ 
the corresponding expected value $ \displaystyle q^{-2n^2} \sum_{T \in \mathcal{T}_n^{\Fr_q}} P(T) $ for $n \geq 2$ is given by 
\begin{align*}
& \frac12 \left[ \left( \frac{ \frac1{q^2} \left(1-\frac1{q^{2n}}\right)\left(1-\frac{1}{q^{2(n-1)}}\right)}{\left(1-\frac1{q^2}\right)\left(1-\frac1{q^4}\right)} \right) - \left( \frac{ \frac1{q^4} \left(1-\frac1{q^{2n}}\right)\left(1-\frac{1}{q^{2(n-1)}}\right)}{\left(1-\frac1{q^2}\right)\left(1-\frac1{q^4}\right)}  \right) \right] \\
& =  \frac{ q^2 \left(1-\frac1{q^{2n}}\right)\left(1-\frac{1}{q^{2(n-1)}}\right)}{2\left(q^4-1\right)} 
\end{align*} 
In the limit as $n$ tends to infinity, this converges to $\displaystyle  \frac{ q^2}{2\left(q^4-1\right)}.$
\end{proof}

{\small
\bibliographystyle{amsalpha}
\bibliography{referFI-MaxTori}

\noindent Rita Jim\'enez Rolland \\ 
 Instituto de Matem\'aticas, Universidad Nacional Aut\'onoma de M\'exico\\  
Oaxaca de Ju\'arez, Oaxaca,  M\'exico 68000\\
{\texttt{rita@im.unam.mx}}\\

\noindent  Jennifer C. H. Wilson \\
Stanford University,  Department of Mathematics\\
Sloan Hall, 450 Serra Mall, Building 380,  Stanford, CA 94305\\
{\texttt{jchw@stanford.edu}}
}

\end{document}